\documentclass[a4paper,12pt,reqno]{amsart}
\usepackage[T1]{fontenc}
\usepackage[utf8]{inputenc}
\usepackage{lmodern}
\usepackage[english]{babel}
\usepackage{csquotes}

\usepackage[%
left=2.5cm,       
right=2.5cm,      
top=3.5cm,        
bottom=3.5cm,     
heightrounded,    
bindingoffset=0mm 
]{geometry}

\usepackage{amsmath}
\usepackage{amssymb}
\usepackage{mathtools}
\usepackage{mathrsfs}
\usepackage[normalem]{ulem}
\usepackage[hypertexnames=false,colorlinks=true,allcolors=olive]{hyperref} 
\usepackage[noabbrev,capitalize]{cleveref}
\usepackage{bookmark}
\usepackage[dvipsnames]{xcolor}
\usepackage{graphicx}
\usepackage{enumitem}
\usepackage{url}
\usepackage{comment}

\numberwithin{equation}{section}

\theoremstyle{plain}
\newtheorem{theorem}{Theorem}[section]
\newtheorem{lemma}[theorem]{Lemma}
\newtheorem{proposition}[theorem]{Proposition}
\newtheorem{corollary}[theorem]{Corollary}

\theoremstyle{definition}
\newtheorem{definition}[theorem]{Definition}

\newtheorem{remark}[theorem]{Remark}

\newcommand{\N}{\mathbb{N}}
\newcommand{\R}{\mathbb{R}}
\newcommand{\eps}{\varepsilon}

\renewcommand{\phi}{\varphi}
\renewcommand{\rho}{\varrho}
\renewcommand{\theta}{\vartheta}

\DeclareMathOperator*{\distance}{dist}

\DeclareMathOperator{\Lip}{Lip}

\DeclarePairedDelimiter{\set}{\{}{\}}
\DeclarePairedDelimiter{\abs}{|}{|}
\DeclarePairedDelimiter{\norm}{\|}{\|}

\allowdisplaybreaks

\label{BEGIN:newcommand}
\newcommand{\dimension}{N}
\newcommand{\measurableSetsDim}{\mathcal{M}_{\dimension}}
\newcommand{\volume}{m}
\newcommand{\volumezero}{\volume_{0}}

\newcommand{\Rdim}{\R^{\dimension}}

\newcommand{\de}{\mathrm{d}}
\newcommand{\integralde}{\thinspace\mathrm{d}}
\newcommand{\integral}[1]{\int_{#1}}
\newcommand{\BallRadiusCenter}[2]{B_{#1}({#2})}
\newcommand{\BallVolumeCenter}[2]{B^{#1}({#2})}
\newcommand{\BallRadius}[1]{B_{#1}}
\newcommand{\BallVolume}[1]{B^{#1}}

\DeclareMathOperator{\leb}{\mathcal{L}}
\newcommand{\lebdim}{\leb^{\dimension}}

\newcommand{\omegadim}{\omega_{\dimension}}

\DeclareMathOperator{\diam}{diam}
\newcommand{\Ps}{P_{s}}
\newcommand{\Etilde}{\widetilde{E}}
\newcommand{\EucEnergyParDat}[2]{\mathcal{E}_{#1,#2}} 
\newcommand{\EucEnergyParDatVolume}[3]{\mathcal{E}_{#1,#2}^{#3}}

\newcommand{\Evol}{E_{\volume}}
\newcommand{\indice}{i}
\newcommand{\indiceSet}{\mathcal{I}_{\volume}}
\newcommand{\Evolindice}{E_{\volume}^{\indice}}
\newcommand{\FamilyMinimizersVol}{\mathcal{S}_{\volume}}
\newcommand{\potential}{g}
\newcommand{\volScaledPotential}{{\potential}_{\volume}} 

\newcommand{\EucEnergysPotential}{\EucEnergyParDat{s}{\potential}}
\newcommand{\EucEnergysScaledPotentialVolume}[1]{\EucEnergyParDatVolume{s}{\volScaledPotential}{#1}}

\newcommand{\EtildeVol}{\Etilde_{\volume}}
\newcommand{\EucEnergysVolScaledPotential}{\EucEnergyParDat{s}{\volScaledPotential}}

\newcommand{\Bone}{\BallRadius{1}}
\newcommand{\xtilde}{\tilde{x}}

\DeclareMathOperator*{\argmin}{arg\,min}
\DeclareMathOperator*{\FrkAsym}{A}
\DeclareMathOperator*{\WulDefs}{\delta_s}
\newcommand{\xvolume}{x_{\volume}}
\newcommand{\xvolumeindice}{x_{\volume}^{\indice}}
\DeclarePairedDelimiter{\tonde}{(}{)}
\DeclarePairedDelimiter{\quadre}{[}{]}
\DeclarePairedDelimiter{\graffe}{\{}{\}}
\newcommand{\sequence}[2]{\graffe{#1}_{#2}}

\newcommand{\Lspace}[1]{L^{#1}}

\newcommand{\Rzero}{R_{0}}
\newcommand{\length}{\sigma(\volume)}
\newcommand{\lengthzero}{\sigma(\volumezero)}

\newcommand{\comp}[1]{{#1}^c} 
\newcommand{\HdimMinusOne}{\mathcal{H}^{\dimension-1}}
\newcommand*{\MeanCurvature}[1]{H_{#1}}
\newcommand*{\sMeanCurvature}[1]{H_{s,\partial#1}}
\newcommand*{\rzero}{r_{0}}
\newcommand*{\BFVker}[1]{L_{#1}}
\newcommand*{\Lr}{\BFVker{r}}

\newcommand*{\symmDiff}{\bigtriangleup}

\newcommand*{\penal}{\mu}

\newcommand*{\SetSep}{:}

\newcommand{\dist}{\mathrm{dist}\,}

\newcommand{\capE}{\mathcal{E}}
\newcommand{\spt}{\mathrm{spt}\, }

\label{END:newcommand}
\begin{document}
	
	\title{On the Shape of Small Liquid Drops Minimizing Nonlocal Energies}
	
	\author[K.~Bessas]{Konstantinos Bessas}
	\address[K.~Bessas]{Dipartimento di Matematica, Università di Pavia, Via Adolfo Ferrata 5, 27100 Pavia, Italy}
	\email{konstantinos.bessas@unipv.it}
	
	\author[M.~Novaga]{Matteo Novaga}
	\address[M.~Novaga]{Dipartimento di Matematica, Università di Pisa, Largo Bruno Pontecorvo 5, 56127 Pisa, Italy}
	\email{matteo.novaga@unipi.it}
	
	\author[F.~Onoue]{Fumihiko Onoue}
	\address[F.~Onoue]{Zentrum Mathematik, Technische Universit\"at M\"unchen, Boltzmannstrasse 3, 85748 Garching, Germany}
	\email{fumihiko.onoue@tum.de}

	\date{\today}
	
	
	
	\thanks{\textit{Acknowledgements}. KB and MN are members of INdAM-GNAMPA. The work of KB is partially supported by the INdAM--GNAMPA 2022 Project \textit{Fenomeni non locali in problemi locali}, codice CUP\_E55\-F22\-00\-02\-70\-001. The work of FO was supported by the DFG Collaborative Research Center TRR 109, ``Discretization in Geometry and Dynamics''.
	}
	
	\begin{abstract}
		We study the equilibrium shape of liquid drops minimizing the fractional perimeter under the action of a potential energy. We prove, with a quantitative estimate, that the small volume minimizers are convex and uniformly close to a ball. 
	\end{abstract}
	
	\maketitle
	
	\tableofcontents
	
\section{Introduction}\label{sec:INTRO}
Given $s \in (0,1)$ and $\volume>0$, we study the following minimization problem:
\begin{equation}\label{MinimizationNonlocalProblem}
	\min\left\{ \EucEnergysPotential(E) \coloneqq \Ps(E) + \int_{E}\potential(x)\integralde x \SetSep |E| = \volume \right\},
\end{equation}
where $\Ps(E)$ is the $s$-fractional perimeter of a measurable set $E \subset \Rdim$ defined as
\begin{equation}
	\Ps(E) \coloneqq \int_{E}\int_{E^c} \frac{1}{|x-y|^{\dimension+s}}\integralde x\,\de y,
\end{equation}
and $\int_{E}g(x)\,dx$ is the potential energy of $E$ associated with a locally Lipschitz and coercive function $\potential : \Rdim \mapsto \R$. The notion of  $s$-fractional perimeter was introduced by Caffarelli, Roquejoffre, and Savin in \cite{CRS10}, who were motivated by the classical phase field model when long-range correlations exist. After their celebrated work, variational problems involving the $s$-fractional perimeter have been studied by many authors as an analogy of the classical ones and our paper is in this line of research.

Our research is motivated by the classical minimization problem for the equilibrium shape of liquid drops and crystals. The classical problem can be formulated as follows: 
\begin{equation}\label{MinimizationLocalProblem}
	\min\left\{ \capE(E) \coloneqq \int_{\partial^* E}f(\nu_E) \, d\mathcal{H}^{N-1} + \int_{E}\potential(x)\integralde x \SetSep |E| = \volume \right\},
\end{equation}
where $f : \mathbb{R}^N \to [0,\,\infty)$ is a convex, positively 1-homogeneous function, $\nu_E$ is the measure-theoretic unit normal to $E$, and $\partial^* E$ is the reduced boundary (see the definition of $\nu_E$ and $\partial^* E$ in \cite{Maggi12}). Notice that, if $f$ is the Euclidean norm, then the first term in $\capE$ is exactly the classical perimeter of a set $E \subset \Rdim$ in the sense of De Giorgi. 

Almgren originally proposed the following question on the shape of minimizers for Problem \eqref{MinimizationLocalProblem}. This question is mentioned in \cite{McCane98}. 
\begin{description}
	\item[Question] If the potential $\potential$ is convex, then is it true that the minimizer for Problem \eqref{MinimizationLocalProblem} is convex (at least connected) and unique? 
\end{description}
The geometric properties of minimizers for liquid drop and crystal models associated with Problem \eqref{MinimizationLocalProblem} have been studied for decades by many authors, for instance, Almgren, Avron, Baer, De Philippis, Figalli, Finn, Goldman, Gonzalez, Maggi, Massari, McCann, Tamanini, Taylor, Zia, and so on \cite{Finn80, GMT80, ATZ83, Finn86, FigMag11, Baer15, DePGol18}. Partial answers to Question are given by some of these authors.

In particular, Figalli and Maggi in \cite{FigMag11} extensively studied the shape of minimizers for Problem \eqref{MinimizationLocalProblem} in the situation where the contribution coming from the potential term $E \mapsto \int_{E}\potential(x)\integralde x$ can be negligible compared with the perimeter term associated with $f$. They proved that the minimizers of $\capE$ are uniformly close to a Wulff shape when the volume is sufficiently small. In addition, if $\dimension=2$, they showed that the minimizers are convex and, if $\dimension \geq 3$, assuming a stronger regularity on the 1-homogeneous function $f$ and the potential $g$ of the energy $\capE$, they also proved the convexity of minimizers with a quantitative estimate for the second fundamental form on the boundaries of minimizers. We remark that De Philippis and Goldman in \cite{DePGol18} proved that, if the function $f$ is smooth and uniformly elliptic, and the potential $g$ is convex and coercive, then any minimizer for Problem \eqref{MinimizationLocalProblem} without volume constraint is convex in any dimension. Moreover, under the same assumptions on $f$ and $g$ as above, they showed that, in dimension 2, any minimizer for Problem \eqref{MinimizationLocalProblem} (with volume constraint in this case) is convex and unique. On the other hand, Indrei in \cite{Indrei20} proved the non-existence of minimizers for Problem \eqref{MinimizationLocalProblem} in dimension 2 if the potential $g$ is merely convex and satisfies $g(0)=0$. Moreover, Indrei and Karakhanyan in \cite{InKa23} recently showed the sharp quantitative stability for Problem \eqref{MinimizationLocalProblem} if the function $f=1$ on $\mathbb{S}^{N-1}$ and the potential $g$ is radial and non-decreasing and $g(0)=0$, which implies that the unique minimizer for Problem \eqref{MinimizationLocalProblem} is a ball.

To the best of our knowledge, the shape of minimizers for the nonlocal version of Problem \eqref{MinimizationLocalProblem}, that is, Problem \eqref{MinimizationNonlocalProblem} is not well-understood and there are few references on the problem. For instance, Cesaroni and Novaga in \cite{CesNov17} proved the existence and regularity of minimizers for Problem \eqref{MinimizationNonlocalProblem} in the case that the potential function is either coercive or $\mathbb{Z}^N$-periodic. Our goal in this paper is to study the shape of minimizers for Problem \eqref{MinimizationNonlocalProblem} and to establish a nonlocal version of some of the results shown in \cite{FigMag11}. Precisely, we prove
\begin{theorem}\label{theoremMainShapeMinimizers}
	Let $s \in (0,\,1)$. Assume that the potential $\potential$ satisfies \ref{H:coercive} and \ref{H:locallyLipschitz} (see Section \ref{sec:NotationEtAl} for the assumptions). Then, there exist constants $\volumezero=\volumezero(\dimension,s,\potential)>0$ and $\Rzero=\Rzero(\dimension,\potential)>0$ such that, for any $\volume \in (0,\,\volumezero)$, every minimizer $\Evol$ of $\EucEnergysPotential$ with $|\Evol| = \volume$ satisfies the following: there exists a point $\xvolume \in \Rdim$ such that $|\xvolume| \leq \Rzero$, $\lim\limits_{\volume\to 0}\distance(\xvolume,\graffe{\potential=0})=0$ and  

	\begin{equation*}
		x_m + \BallRadius{\sigma(m)(1-\rzero)} \subset E_m \subset x_m + \BallRadius{\sigma(m)(1+\rzero)},
	\end{equation*}
	for some $r_0 \in (0,\,C\volume^{s^2/(2\dimension^2)})$ and some constant $C>0$ depending only on $\dimension$, $s$, and $\potential$, where we set

	\begin{equation}\label{eq:defOfLength}
		\length\coloneqq\tonde*{\frac{\volume}{|\BallRadius{1}|}}^{\frac{1}{\dimension}}.
	\end{equation}
	
    Moreover, we have that $\partial \Evol$ is of class $C^{2,\alpha}$ with $\alpha < s$ and $\partial E_m$ converges to $\partial\BallRadius{1}$ in $C^2$ as $m \downarrow 0$, by proper scaling and translation. In particular, $\Evol$ is convex for small $\volume \in (0,\,m_0)$.
\end{theorem}

From a geometric point of view, the existence of minimizers for Problem \eqref{MinimizationLocalProblem} is related to the existence of a smooth hypersurface $\mathcal{M}$ on which the following prescribed mean curvature equation holds:
\begin{equation}
	\MeanCurvature{\mathcal{M}} + \potential = 0 \quad \text{on $\mathcal{M}$},
\end{equation}
where $\MeanCurvature{\mathcal{M}}$ is the mean curvature on $\mathcal{M}$. Indeed, if $E_m \subset \mathbb{R}^N$ is a smooth minimizer for Problem \eqref{MinimizationLocalProblem} with $|E_m|=m$, then we can obtain the following Euler-Lagrange equation:
\begin{equation*}
	H_{\partial E_m} + g = \lambda_m \quad \text{on $\partial E_m$}.
\end{equation*}
As a nonlocal analogy of the classical problem, the existence of minimizers of our energy $\EucEnergysPotential$ is related to the problem of finding a set $E$ with a smooth boundary such that the following prescribed ``fractional'' mean curvature equation holds on $\partial E$: 
\begin{equation}\label{geometricEquationFracMc}
	\sMeanCurvature{E}(x) + \potential(x) = 0 \quad \text{for $x \in \partial E$}.
\end{equation}
Here $\sMeanCurvature{E}(x)$ is the $s$-fractional mean curvature on $\partial E$ at $x$, which is defined as
\begin{equation}
	\sMeanCurvature{E}(x) \coloneqq \text{P.V.} \int_{\Rdim} \frac{\chi_{E^c}(y) - \chi_{E}(y)}{|y-x|^{\dimension+s}}\,\de y
\end{equation}
where ``P.V.'' means the Cauchy principle value. Indeed, assuming that $\Evol$ is a minimizer of $\EucEnergysPotential$ among sets with volume $\volume>0$ and $\partial \Evol$ is smooth, we can show that the following Euler-Lagrange equation holds true:
\begin{equation}\label{eulerLagrangeEqFrMc}
	\sMeanCurvature{\Evol} + \potential = \lambda_m \quad \text{on $\partial \Evol$}
\end{equation}
where $\lambda_m$ is a Lagrange multiplier. As we mentioned above, the authors in \cite{CesNov17} proved the existence of bounded minimizers with smooth boundary for Problem \eqref{MinimizationNonlocalProblem} and this implies the existence of compact hypersurfaces $\mathcal{M} \subset \mathbb{R}^N$ on which the equation \eqref{eulerLagrangeEqFrMc} holds true for any $m>0$. In this context, our main theorem may imply the existence of a compact, smooth hypersurface $\mathcal{M}_m \subset \mathbb{R}^N$ with non-negative mean curvature such that the equation \eqref{eulerLagrangeEqFrMc} holds on $\mathcal{M}_m$ for small $m>0$ and $\mathcal{M}_m$, by properly rescaling, converges to the sphere in $\mathbb{R}^N$ as $m \downarrow 0$.

The proof of \cref{theoremMainShapeMinimizers} is divided in several steps. The former part of Theorem \ref{theoremMainShapeMinimizers} (the uniform closeness of minimizers to the Euclidean ball) will be proved in Section \ref{sectionUniformClosenessToBall} and the latter part of Theorem \ref{theoremMainShapeMinimizers} (the regularity and convexity of minimizers) will be proved in Section \ref{sectionRegularityConvexityMinimizers}.

The organization of our paper is as follows: in \cref{sec:NotationEtAl} we fix some notation and we introduce the mathematical setting of our problem. 

In \cref{sectionUniformClosenessToBall} we first establish properties of minimizers of \eqref{eq:defOfEm} and consider a suitable rescaling for them. Then, we prove the uniform closeness of the rescaled minimizers to the unit Euclidean ball for small volumes in \cref{thm:Uniform_closeness_of_minimizers_to_ball}.
The strategy follows some arguments in \cite{CesNov17,DNRV15,FigMag11}.
Precisely, we derive suitable integro-differential inequalities for estimating the volume of the difference between rescaled minimizers and balls centred in the origin.
A similar argument can be found in \cite[Proposition 3.2]{CesNov17}, where the authors, applying a nonlocal version of the so-called Almgren’s Lemma (see \cite[Lemma 3]{FigMag11}), prove the boundedness of minimizers of \eqref{eq:defOfEm} for a fixed value of the volume. The first main difference with the approach in \cite{CesNov17} is that our estimates not only aim at obtaining boundedness of minimizer, but also at closeness in $\Lspace{\infty}$ to the unit ball (for small volumes).
The second one is that our estimates are uniform with respect to the volume of minimizers.
To obtain the uniformity with respect to the volume, we adapt ideas in \cite{FigMag11} coming from the classical case.
In particular, we do not employ a nonlocal Almgren’s Lemma such as \cite[Lemma 3.1]{CesNov17}, but we made use of a suitable dilation. Indeed, with our approach, the constants involved in our estimates depend only on the volume of the minimizers rather than on the minimizers themselves.  

In \cref{sec:UncoMinProblem}, with the uniform proximity result in our hands,  we can reformulate Problem \eqref{eq:defOfEm} into a variational unconstrained problem, which becomes crucial for establishing regularity. 

In \cref{sec:Regular_Convex_Minimizers}, we obtain the regularity and convexity of minimizers for Problem \eqref{MinimizationNonlocalProblem}, employing several results on the regularity of the so-called \enquote{almost minimizers} or ``$\Lambda$-minimizers'', where $\Lambda$ is independent of minimizers, of the $s$-fractional perimeter and the bootstrap argument for integro-differential equations shown in \cite{BFV14}. As a consequence, by applying the regularity result for the ``$\Lambda$-minimizers'' shown in \cite[Corollary 3.6]{FFMMM15} and from the uniform closeness of the minimizers to balls that we prove in Section \ref{sectionUniformClosenessToBall}, we obtain the convergence of the minimizers to the Euclidean ball in $C^2$-topology, which implies in particular the convexity of minimizers.

\section{Notation and Setting of the Problem}\label{sec:NotationEtAl}

We denote the Euclidean norm of $x\in\Rdim$ by $|x|$, the Euclidean ball of radius $r>0$ and centre $x\in\Rdim$ by $\BallRadiusCenter{r}{x}$ and the Euclidean ball of volume $m>0$ and centre $x\in\Rdim$ by $\BallVolumeCenter{m}{x}$. We also set $\BallRadius{r}\coloneqq\BallRadiusCenter{r}{0}$ and $\BallVolume{m}\coloneqq\BallVolumeCenter{m}{0}$.

If $E$ is a non-empty subsets of $\Rdim$, then $\dist(x,E)\coloneqq\inf\graffe*{|x-y|\SetSep y\in E}$ for all $x\in\Rdim$.

If $E,F$ are two subsets of $\Rdim$ we indicate by $E\symmDiff F$ their \textit{symmetric difference}, i.e $E\symmDiff F\coloneqq (E\setminus F)\cup(F\setminus E)$.

We denote the Lebesgue measure of $\Rdim$ by $\lebdim$ and the class of all Lebesgue measurable sets of $\Rdim$ by $\measurableSetsDim$. For $E\in\measurableSetsDim$ we also set $|E|\coloneqq\lebdim(E)$ and we denote the topological boundary of $E\in\measurableSetsDim$ by $\partial E$. $\HdimMinusOne$ denotes the $(\dimension-1)-$dimensional Hausdorff measure of $\Rdim$.

Let $\potential:\Rdim\to\R$ be
\begin{enumerate}[label=(H\arabic*)]
\item\label{H:coercive} coercive, i.e. s.t. $\lim\limits_{|x|\to+\infty}\potential(x)=+\infty$;
\item\label{H:locallyLipschitz} locally Lipschitz.
\end{enumerate}

Let $\Evol$ be a minimizer of the constrained problem \eqref{MinimizationNonlocalProblem} with volume $\volume$, that is

\begin{equation}\label{eq:defOfEm}
	\Evol\in\argmin\set*{\EucEnergysPotential(E) \coloneqq \Ps(E) + \int_{E}\potential(x)\integralde x \SetSep |E| = \volume}.
\end{equation}

Moreover, we recall the definition of \eqref{eq:defOfLength}	
for every $\volume>0$, so that $\BallVolume{\volume}=\BallRadius{\length}$ (the ball of volume $\volume$ is equal to the ball of radius $\length$).

Without loss of generality (thanks to the structure of the volume constrained problem in \eqref{eq:defOfEm}) in the whole work we assume 
\begin{equation}\label{eq:Minofg}
	\inf_{\Rdim}\potential=\potential(0)=0.
\end{equation}

\newcommand*{\xhat}{\hat{x}}
\newcommand*{\potentialTranslated}{\hat{\potential}}
Indeed, fixed any $\xhat\in\Rdim$ such that $\potential(\xhat)=\inf_{\Rdim}\potential$, by replacing the potential $\potential$ with $\potentialTranslated$, defined by $\potentialTranslated(x)\coloneqq\potential(x+\xhat)-\inf_{\Rdim}\potential$, on the one hand we observe that $\potentialTranslated$ satisfies \ref{H:coercive}, \ref{H:locallyLipschitz} and \eqref{eq:defOfEm} on the other one we note that the minimizers of $\EucEnergyParDat{s}{\potentialTranslated}$ with volume $\volume$ are the same minimizers of $\EucEnergysPotential$ with volume $\volume$ (upon a translation via $\xhat$).

\begin{remark}

We observe that $\Evol$ is well-defined, since the problem \eqref{MinimizationNonlocalProblem} admits solution if the potential $\potential$ is measurable, bounded from below and satisfies assumption \ref{H:coercive}. This can be proved via the Direct Method as done in \cite[Proposition 5.3]{CesNov17}.

As far as uniqueness is concerned, to the best of our knowledge, there are no complete results available in the literature. Even in the classical case (see Problem \eqref{MinimizationLocalProblem}) a general answer has not been provided yet, even though there are some partial results. Precisely, from \cite[Remark 1.6]{DePGol18} it follows that for large volumes, minimizers of \eqref{MinimizationLocalProblem} are unique when the potential $\potential$ is convex and coercive.

We also mention that in \cite[Theorem 1.1]{McCane98} uniqueness (up to translations) is established in $2D$, for any volume but among convex competitors, when the potential $\potential$ is convex and the set where it vanishes is bounded but non-empty.

\end{remark}

For every $\volume>0$, we let 
\begin{equation}\label{eq:defOfFamilyMinimizersVol}
\FamilyMinimizersVol\coloneqq\set*{\Evolindice}_{\indice\in\indiceSet}
\end{equation}
be the family of all the minimizers of $\EucEnergysPotential$ with volume constraint $|\Evolindice|=\volume$.
We observe that $\FamilyMinimizersVol\neq\emptyset$ because there always exists a minimizer of $\capE_{s,\potential}$ for every $m>0$.
In the sequel, by a little abuse of notation, we will drop the dependence on $\indice\in\indiceSet$ in $\Evolindice$ when it is not essential to highlight it (and also in other quantities possibly depending on $\indice\in\indiceSet$).

\section{Uniform Proximity to the Unit Ball} \label{sectionUniformClosenessToBall}

\begin{definition}\label{def:FraenkelANDsWulff}
	Let $E\in\measurableSetsDim$, i.e a measurable set of $\Rdim$, such that $0<|E|<+\infty$. We define the
	\textit{Fraenkel asymmetry of} $E$ as
	\begin{equation*}
		\FrkAsym(E)\coloneqq\inf \set*{\frac{|E\symmDiff\BallVolumeCenter{|E|}{x}|}{|E|}\SetSep x\in\Rdim},
	\end{equation*}
where $\BallVolumeCenter{|E|}{x}$ is the ball centred at $x$ with volume $|E|$.
Moreover, if $\Ps(E)<+\infty$ we define the \textit{Wulff $s$-deficit of} $E$ as
\begin{equation*}
	\WulDefs(E):=\frac{\Ps(E)}{\Ps(\BallVolume{|E|})}-1,
\end{equation*}
where $\BallVolume{|E|}$ is the ball centred at the origin with volume $|E|$.
\end{definition}
\begin{remark}
	We observe that $\FrkAsym$ and $\WulDefs$ are scaling and translation invariant, that is: 
	\begin{equation*}
	\FrkAsym(E)=\FrkAsym(\lambda E + x),
	\end{equation*}
	\begin{equation*}
	\WulDefs(E)=\WulDefs(\lambda E + x),
	\end{equation*}
	for every $\lambda\in\R\setminus\{0\}$, $x\in\Rdim$ and $E\in\measurableSetsDim$ with $0<|E|<+\infty$. Furthermore, $\bar{x}$ attains the infimum in the definition of $\FrkAsym(E)$ if and only if $x+\bar{x}$ attains the infimum in the definition of $\FrkAsym(\lambda E + x)$.
\end{remark}

We state the following quantitative isoperimetric inequality for the fractional perimeter $\Ps$ (see also \cite{FFMMM15}).

\begin{theorem}\label{thm:QuantIsoPs}
	Let $\dimension\geq2$ and $s\in(0,1)$.	
	Let $E\in\measurableSetsDim$, i.e a measurable set of $\Rdim$, such that $0<|E|<+\infty$. Then, there exists $C(\dimension,s)>0$ s.t.
	\begin{equation*}
		\WulDefs(E)\geq C(\dimension,s) \FrkAsym(E)^2.
	\end{equation*}
\end{theorem}

We are now ready to establish properties of minimizers of $\eqref{MinimizationNonlocalProblem}$.

\begin{proposition}[Properties of Minimizers]\label{prop:properties_of_minimizers}
Let $s\in(0,1)$ and assume that $\potential$ satisfies \ref{H:coercive}, \ref{H:locallyLipschitz} and \eqref{eq:Minofg}.
Let 
$\FamilyMinimizersVol\coloneqq\set*{\Evolindice}_{\indice\in\indiceSet}$
be as in \eqref{eq:defOfFamilyMinimizersVol} and $\length$ be as in \eqref{eq:defOfLength} for every $\volume>0$.
Then, there exists $C(\dimension,s)>0$ s.t. for every $\volume>0$ and $\indice\in\indiceSet$
\begin{equation}\label{eq:WulDefsEstimate}
	\WulDefs(\Evolindice)\leq C(\dimension,s)\length^s\sup_{\BallRadius{\length}}\potential,
\end{equation}

\begin{equation}\label{eq:FrkAsymEstimate}
	\FrkAsym(\Evolindice)=\frac{|\Evolindice\symmDiff\BallVolumeCenter{\volume}{\xvolumeindice}|}{\volume}\leq C(\dimension,s)\sqrt{\length^s\sup_{\BallRadius{\length}}\potential},
\end{equation}
for some $\xvolume=\xvolumeindice\in\Rdim$.

Moreover, there exist $\volumezero=\volumezero(\dimension, s,\potential)>0$ and $\Rzero=\Rzero(\dimension, \potential)>0$ such that for every $\volume\leq\volumezero$ and $\indice\in\indiceSet$
\begin{equation}\label{eq:translationsUB}
	|\xvolumeindice|\leq\Rzero,
\end{equation}
and
\begin{equation}\label{eq:translationsConverging}
\lim\limits_{\volume\to 0}\sup_{\indice\in\indiceSet}\distance(\xvolumeindice,\graffe{\potential=0})=0.
\end{equation}
Furthermore, if we let
\begin{equation}\label{eq:Def_Rescaled_Minimizer}
	\EtildeVol\coloneqq\frac{(\Evol-\xvolume)}{\length}
\end{equation}
and
\begin{equation}\label{eq:Def_VolScaledPotential}
	\volScaledPotential(x)\coloneqq\length^s \potential(\length x + \xvolume),
\end{equation}	
then
	\begin{equation}\label{eq:RescaledSetSolvesRescPB}
		\EtildeVol\in\argmin\left\{ \Ps(F) + \int_{F}\volScaledPotential(x)\integralde x \SetSep |F| = |\Bone| \right\}
	\end{equation}
for every $\volume>0$.
\end{proposition}
\begin{proof}
	
	This proof follows the strategy adopted in part of \cite[Theorem 9]{FigMag11} combined with \cite[Proposition 3.2]{CesNov17}.
	
	Let $\volume>0$ and $\indice\in\indiceSet$.
	By the non-negativity  of $\potential$ (recall \eqref{eq:Minofg}) and the minimality of $\Evolindice$ (see \eqref{eq:defOfFamilyMinimizersVol}),
	\begin{equation}\label{key}
		\Ps(\Evolindice)\leq\Ps(\Evolindice)+\integral{\Evolindice}\potential\integralde x\leq \Ps(\BallVolume{\volume})+\integral{\BallVolume{\volume}}\potential\integralde x.
	\end{equation}
	Then,
	\begin{align*}
		\WulDefs(\Evolindice)\leq\frac{\integral{\BallVolume{\volume}}\potential(x)\integralde x}{\Ps(\BallVolume{\volume})}=\frac{\length^{\dimension}\integral{\Bone}\potential(\length \xtilde)\integralde \xtilde}{\length^{\dimension-s}\Ps(\Bone)}\leq \frac{|\Bone|}{\Ps(\Bone)}\length^s\sup_{\BallRadius{\length}}\potential,
	\end{align*}
so that \eqref{eq:WulDefsEstimate} is proved.
Since by \cite[Proposition 3.2]{CesNov17} $\Evolindice$ is bounded, the infimum in the definition of $\FrkAsym(\Evolindice)$ (see \cref{def:FraenkelANDsWulff}) is actually a minimum. This observation joint with \cref{thm:QuantIsoPs} and \eqref{eq:WulDefsEstimate} implies \eqref{eq:FrkAsymEstimate} for some $\xvolumeindice\in\Rdim$.

Thanks to the local boundedness of $\potential$ (see \ref{H:locallyLipschitz}), from \eqref{eq:FrkAsymEstimate} we can find a constant $\volumezero(\dimension, s, \potential)>0$ such that
\begin{equation*}
	\frac{|\Evolindice\setminus\BallRadiusCenter{\length}{\xvolumeindice}|}{\volume}\leq\frac{1}{2},
\end{equation*} 
for all $\volume\leq\volumezero$ and $\indice\in\indiceSet$;
which (since $|\Evolindice|=\volume$) is equivalent to
\begin{equation}\label{k1}
	|\Evolindice\cap\BallRadiusCenter{\length}{\xvolumeindice}|\geq\frac{\volume}{2},
\end{equation} 
for all $\volume\leq\volumezero$  and $\indice\in\indiceSet$.

\newcommand{\distxvolToZeroset}{d_m}
\newcommand{\badsetforconvergence}{I_{\volumezero}}

Since \ref{H:coercive} holds, we observe that $\graffe*{\potential=0}$ is contained in a compact subset of $\Rdim$. Then, \eqref{eq:translationsConverging} implies \eqref{eq:translationsUB} for some $\Rzero=\Rzero(\dimension,\potential)>0$ (possibly choosing $\volumezero=\volumezero(\dimension, s,\potential)>0$ smaller).

We now claim that \eqref{eq:translationsConverging} holds.

We define $\distxvolToZeroset\coloneqq\sup_{\indice\in\indiceSet}\distance(\xvolumeindice,\graffe{\potential=0})$ for every $\volume\in (0,\volumezero]$ and we consider the set $\badsetforconvergence\coloneqq\graffe{\volume\in (0,\volumezero]:\distxvolToZeroset>2\length}$.
	
	We first observe that the fact that $\length$ converges to zero as $\volume\to 0$ implies,
	\begin{equation}\label{eq:Trivial}
	\lim_{\substack{\volume\to 0\\\volume\in(0,\volumezero]\setminus\badsetforconvergence}}\distxvolToZeroset =0,
	\end{equation}
	whenever $0$ is an accumulation point for $(0,\volumezero]\setminus\badsetforconvergence$.
	
	If $\volume\in\badsetforconvergence$, then $\BallRadiusCenter{\length}{\xvolumeindice}\subset\Rdim\setminus\tonde{\graffe{\potential=0}+\BallRadius{\distxvolToZeroset-2\length}}$ for some $\indice\in\indiceSet$ (by definition of $\distxvolToZeroset$ and the triangular inequality).
	
	We now observe that by the fractional isoperimetric inequality and \eqref{key} it immediately follows
	\begin{equation}\label{k2}
		\integral{\Evolindice}\potential\integralde x\leq\integral{\BallVolume{\volume}}\potential\integralde x.
	\end{equation}
	Then, by \eqref{k1}, \eqref{k2} and the non-negativity  of $\potential$ (see \eqref{eq:Minofg}):
	\begin{equation*}
		\frac{\volume}{2}\inf_{\Rdim\setminus\tonde{\graffe{\potential=0}+\BallRadius{\distxvolToZeroset-2\length}}}\potential\leq\integral{\Evolindice\cap\BallRadiusCenter{\length}{\xvolumeindice}}\potential\integralde x\leq \integral{\BallVolume{\volume}}\potential\integralde x\leq\volume\sup_{\BallRadius{\length}}\potential,
	\end{equation*}
	for some $\indice\in\indiceSet$.
	
	Therefore, since $\potential$ is continuous and $\potential(0)=0$ (see \eqref{eq:Minofg} and \ref{H:locallyLipschitz}):
	\begin{equation}\label{eq:ContImpliesxmconv}
			\limsup_{\substack{\volume\to 0\\\volume\in\badsetforconvergence}}\inf_{\Rdim\setminus\tonde{\graffe{\potential=0}+\BallRadius{\distxvolToZeroset-2\length}}}\potential\leq 2 \limsup_{\volume\to 0}\sup_{\BallRadius{\length}}\potential=0,
		\end{equation}
		
	whenever $0$ is an accumulation point for $\badsetforconvergence$.
	
	From \eqref{eq:ContImpliesxmconv} and the coercivity of $\potential$ (see \ref{H:coercive}) and the non-negativity  of $\potential$ (see \eqref{eq:Minofg}) we then deduce that
	
	\begin{equation}\label{eq:LessTrivial}
		\lim_{\substack{\volume\to 0\\\volume\in\badsetforconvergence}}\distxvolToZeroset =0,
		\end{equation}
	whenever $0$ is an accumulation point for $\badsetforconvergence$.
	
Therefore, \eqref{eq:translationsConverging} is proved by combining \eqref{eq:Trivial} and \eqref{eq:LessTrivial}.

Since
\begin{equation*}
	\EucEnergysPotential(\length F+\xvolume)=\length^{\dimension-s}\EucEnergysVolScaledPotential(F)
\end{equation*}
for every $F\in\measurableSetsDim$ with $|F|=|\Bone|$, \eqref{eq:RescaledSetSolvesRescPB} follows.
\end{proof}

In order to prove the uniform closeness of the rescaled minimizers $\EtildeVol$ to the Euclidean ball for small volumes (see \cref{thm:Uniform_closeness_of_minimizers_to_ball}), we need two technical lemmas.

\begin{lemma}[Estimate for the Fractional Derivative]\label{lem:estimateForFracDerivative}
	Let $a,b\in\R$ such that $a<b$ and $s\in(0,1)$. 
	\begin{enumerate}[label=(\roman*)]
		\item\label{item:decreasing_case} Let $f:[a,+\infty)\to [0,+\infty)$ be absolutely continuous and monotone non-increasing such that $\lim\limits_{r\to+\infty}f(r)=0$.
		Then,
		\begin{equation}\label{eq:stimFracDer}
			-(1-s)\int_{a}^{b}\integralde r \int_{r}^{+\infty}\integralde t \frac{f'(t)}{(t-r)^s}\leq (b-a)^{1-s}f(a)
		\end{equation}
		if the integral in \eqref{eq:stimFracDer} is finite.
		
		\item\label{item:increasing_case} Let $f:[0,b]\to [0,+\infty)$ be absolutely continuous and monotone non-decreasing such that $f(0)=0$.
		Then,
		\begin{equation}\label{eq:stimFracDerBIS}
			(1-s)\int_{a}^{b}\integralde r \int_{0}^{r}\integralde t \frac{f'(t)}{(r-t)^s}\leq (b-a)^{1-s}f(b)
		\end{equation}
		if the integral in \eqref{eq:stimFracDerBIS} is finite.
		
	\end{enumerate}

\end{lemma}
\begin{proof}
	We prove only $\ref{item:decreasing_case}$ since the proof of $\ref{item:increasing_case}$ is analogous.
	
	By Fubini's Theorem:
	\begin{align*}
			&-(1-s)\int_{a}^{b}\integralde r \int_{r}^{+\infty}\integralde t \frac{f'(t)}{(t-r)^s}=
			\int_{a}^{+\infty}f'(t)\integralde t \int_{a}^{t\wedge b}\integralde r\frac{\de}{\de r}\quadre*{(t-r)^{1-s}}\\
			&=\int_{a}^{+\infty}f'(t)\quadre*{(t-t\wedge b)^{1-s}-(t-a)^{1-s}}\integralde t\\		
			&=\int_{a}^{b}(-f'(t))(t-a)^{1-s}\integralde t+
			\int_{b}^{+\infty}(-f'(t))\quadre*{(t-a)^{1-s}-(t-b)^{1-s}}\integralde t\\
			&\leq(b-a)^{1-s}\int_{a}^{+\infty}(-f'(t))=(b-a)^{1-s}f(a)
	\end{align*}
We also used that the function $(b,+\infty)\ni t\mapsto (t-a)^{1-s}-(t-b)^{1-s}$ is non-negative and monotone non-increasing, so it can be bounded by $(b-a)^{1-s}$.

\end{proof}

\begin{lemma}\label{lem:CompLikeDNRV}
	\newcommand{\cConst}{c}
	\newcommand{\ufunction}{u}
	\newcommand{\wfunction}{w}
	\newcommand{\aReal}{a}
	\newcommand{\bReal}{b}
	Let $a,b\in\R$ such that $a<b$, $s\in(0,1)$ and $\cConst>0$.	 
	
	\begin{enumerate}[label=(\roman*)]
		\item\label{item:DNRVdecreasing_case} Let $\ufunction:[\aReal,\bReal]\to [0,+\infty)$ be monotone non-increasing such that
		\begin{align*}
			\aReal+\frac{(2\ufunction(\aReal))^{\frac{s}{\dimension}}\dimension}{s\cConst}\leq\bReal,
		\end{align*}
		and
		\begin{align}\label{eq:DNRV_INTDIFF}
			2\cConst\int_{\rho}^{\bReal}\ufunction(r)^{\frac{\dimension-s}{\dimension}}\integralde r\leq \ufunction(\rho),
		\end{align}
		for every $\aReal<\rho<\bReal$.

		Then, \begin{equation}
			\ufunction\tonde*{\aReal+\frac{(2\ufunction(\aReal))^{\frac{s}{\dimension}}\dimension}{s\cConst}}=0
		\end{equation}
		
		\item\label{item:DNRVincreasing_case} Let $\wfunction:[\aReal,\bReal]\to [0,+\infty)$ be monotone non-decreasing such that
				\begin{align*}
					\bReal-\frac{(2\wfunction(\bReal))^{\frac{s}{\dimension}}\dimension}{s\cConst}\geq\aReal,
				\end{align*}
				and
				\begin{align*}
					2\cConst\int_{\aReal}^{\rho}\wfunction(r)^{\frac{\dimension-s}{\dimension}}\integralde r\leq \wfunction(\rho),
				\end{align*}
				for every $\aReal<\rho<\bReal$.

				Then, \begin{equation}
					\wfunction\tonde*{\bReal-\frac{(2\wfunction(\bReal))^{\frac{s}{\dimension}}\dimension}{s\cConst}}=0
				\end{equation}
	\end{enumerate}
	
\end{lemma}
\begin{proof}
	\newcommand{\cConst}{c}
	\newcommand{\ufunction}{u}
	\newcommand{\aReal}{a}
	\newcommand{\bReal}{b}
	The proof follows the argument in \cite[Lemma 4.1]{DNRV15} with slight adjustments.
	
	We prove $\ref{item:DNRVdecreasing_case}$.
	
	We highlight that in this proof the constant $\cConst$ in \eqref{eq:DNRV_INTDIFF} cannot vary from line to line: it has to be considered fixed.
	Now we argue in a similar fashion as in \cite[Lemma 4.1]{DNRV15}.
	\newcommand{\auxODEsol}{h}
	We define the auxiliary function $\auxODEsol:[\aReal,+\infty)\to\R$ as
	\begin{equation*}
		\auxODEsol(\rho)\coloneqq\begin{dcases*}
			\quadre*{(2\ufunction(\aReal))^{\frac{s}{\dimension}}-\cConst\frac{s}{\dimension}(\rho-\aReal)}^{\frac{\dimension}{s}} \text{ if } \aReal\leq\rho\leq \aReal+\frac{(2\ufunction(\aReal))^{\frac{s}{\dimension}}\dimension}{s\cConst}\\
			0 \text{ if } \rho>\aReal+\frac{(2\ufunction(\aReal))^{\frac{s}{\dimension}}\dimension}{s\cConst}
		\end{dcases*}
	\end{equation*}
	for every $\rho\geq\aReal$. We observe that $\auxODEsol$ is continuous, $\auxODEsol(\aReal)=2\ufunction(\aReal)$ and it satisfies
	\begin{align}\label{eq:INTDIFFauxx}
		\cConst\int_{\rho}^{\bReal}\auxODEsol(r)^{\frac{\dimension-s}{\dimension}}\integralde r= \auxODEsol(\rho),
	\end{align}
	for every $\rho\geq\aReal$. 
	Our goal is now to prove that 
	\begin{equation}\label{eq:DNRVGoal}
		\ufunction(\rho)\leq\auxODEsol(\rho)\hspace*{1cm}\forall \rho:\aReal\leq\rho\leq \aReal+\frac{(2\ufunction(\aReal))^{\frac{s}{\dimension}}\dimension}{s\cConst}.
	\end{equation}
	\newcommand{\SETufnAboveAux}{I}
	\newcommand{\infSETufnAboveAux}{R_{*}}
	\newcommand{\Renne}{R_{n}}
	We define $\SETufnAboveAux\coloneqq\set*{\rho\in(\aReal,\bReal):\ufunction\geq\auxODEsol \text{ in } [\rho,\bReal]}$ and we observe that it is an interval and that
	\begin{equation*}
		\aReal+\frac{(2\ufunction(\aReal))^{\frac{s}{\dimension}}\dimension}{s\cConst}\in\SETufnAboveAux.
	\end{equation*}
	Now, we let
	\begin{equation}\label{eq:infcompresoo}
		\infSETufnAboveAux\coloneqq\inf \SETufnAboveAux\in \quadre*{\aReal,\aReal+\frac{(2\ufunction(\aReal))^{\frac{s}{\dimension}}\dimension}{s\cConst}}.
	\end{equation}
	Therefore, we can find a sequence $\sequence{\Renne}{n\in\N}\subseteq[\aReal,\infSETufnAboveAux]$ converging to $\infSETufnAboveAux$, such that $\auxODEsol(\Renne)\geq\ufunction(\Renne)$ for every $n\in\N$. So, combining \eqref{eq:DNRV_INTDIFF} and \eqref{eq:INTDIFFauxx}, we get
	\begin{align*}
		&\auxODEsol(\Renne)\geq\ufunction(\Renne)\geq2\cConst\int_{\Renne}^{\bReal}\ufunction(r)^{\frac{\dimension-s}{\dimension}}\integralde r\geq 2\cConst\int_{\Renne}^{\infSETufnAboveAux}\ufunction(r)^{\frac{\dimension-s}{\dimension}}\integralde r + 2\cConst\int_{\infSETufnAboveAux}^{\bReal}\auxODEsol(r)^{\frac{\dimension-s}{\dimension}}\integralde r\\
		&=2\cConst\int_{\Renne}^{\infSETufnAboveAux}\ufunction(r)^{\frac{\dimension-s}{\dimension}}\integralde r + 2\auxODEsol(\infSETufnAboveAux)\geq 2\auxODEsol(\infSETufnAboveAux).
	\end{align*}
	Letting $n$ go to $+\infty$, we obtain $\auxODEsol(\infSETufnAboveAux)\geq\lim\limits_{r\to\infSETufnAboveAux^-}\ufunction(r)\geq2\auxODEsol(\infSETufnAboveAux)$, which implies $\ufunction(\infSETufnAboveAux)=\auxODEsol(\infSETufnAboveAux)=0$. 
	By definition of $\auxODEsol$ and thanks to \eqref{eq:infcompresoo}, we finally observe that $\infSETufnAboveAux=\aReal+\frac{(2\ufunction(\aReal))^{\frac{s}{\dimension}}\dimension}{s\cConst}$. The proof of $\ref{item:DNRVdecreasing_case}$ is complete. 

The proof of $\ref{item:DNRVincreasing_case}$ is similar to the proof of $\ref{item:DNRVdecreasing_case}$ and thus is left to the reader.

\end{proof}

\begin{@empty}
	\newcommand{\minimizer}{E}
\begin{theorem}[Uniform Closeness of Rescaled Minimizers to $\BallRadius{1}$]\label{thm:Uniform_closeness_of_minimizers_to_ball}
	Let $s\in(0,1)$ and $\potential$ s.t. \ref{H:coercive}, \ref{H:locallyLipschitz} and \eqref{eq:Minofg} hold. For every $\volume>0$ let $\Evol$ be a minimizer of $\EucEnergysPotential$ with $|\Evol|=\volume$ and let $\EtildeVol$ be as in \eqref{eq:Def_Rescaled_Minimizer}.
	There exist $\volumezero=\volumezero(\dimension, s,\potential)>0$ and $C=C(\dimension,s,\potential)>0$ such that if $0<\volume\leq\volumezero(\dimension, s,\potential)$, then 
	$\EtildeVol$ satisfies
	\begin{equation*}
		\BallRadius{1-\rzero}\subset\EtildeVol\subset\BallRadius{1+\rzero},
	\end{equation*}
	for some $0<\rzero\leq C(\dimension,s,\potential)\volume^{s^2/(2\dimension^2)}$. 

\end{theorem}
\begin{proof}
	\newcommand{\soglia}{\eps_{0}(\dimension,s,\potential)}
	\newcommand{\Cdimspotential}{C(\dimension,s,\potential)}
	\newcommand{\CdimspotentialUNCHG}{c}
	\newcommand{\Cdims}{C(\dimension,s)}
	\newcommand{\Cdim}{C(\dimension)}

	The proof is in the spirit of the one of \cite[Theorem 5]{FigMag11}.
	
	We choose  $\volumezero=\volumezero(\dimension, s,\potential)$  smaller than the constant $\volumezero$ in the statement of \cref{prop:properties_of_minimizers} so that 
	\begin{equation}\label{eq:technicalSmallLength}
		\lengthzero<1,
	\end{equation} and we choose $\Rzero=\Rzero(\dimension,\potential)$ and $\xvolume$ as in the statement of \cref{prop:properties_of_minimizers} (so that $|\xvolume|\leq\Rzero$ whenever $\volume\in(0,\volumezero)$, as in \cref{prop:properties_of_minimizers}).
	
	Then, we fix $0<\volume\leq\volumezero(\dimension, s,\potential)$.
	
	For convenience, we put $\minimizer\coloneqq\EtildeVol$ and we change our focus from the variable $\volume$ to $\eps\coloneqq\sqrt{\length^s}$.
	Therefore, applying \cref{prop:properties_of_minimizers}, we deduce that there exists $\soglia>0$ such that  $0<\eps\leq\soglia$ and
	\begin{equation}\label{eq:Wssmall}
		\WulDefs(\minimizer)\leq \Cdimspotential\eps^2;
	\end{equation}
	\begin{equation*}
		|\minimizer\symmDiff\BallRadius{1}|\leq \Cdimspotential\eps.
	\end{equation*}

	We claim (possibly choosing a smaller $\soglia$ or, equivalently, a smaller $\volumezero(\dimension,s,\potential)$) that there exists a constant $\Cdimspotential$ such that
	\begin{equation*}
		\BallRadius{1-\rzero}\subset\minimizer\subset\BallRadius{1+\rzero},
	\end{equation*}
	for some $\rzero=\rzero(\eps,\dimension,s,\potential)\leq\Cdimspotential\eps^{s/\dimension}$.
	From $\eqref{eq:Wssmall}$ (in alternative we can simply use the minimality, testing $\minimizer$ against $\BallRadius{1}$) we deduce
	\begin{equation}\label{eq:unifboundPSmin}
		\Ps(\minimizer)\leq \Cdimspotential.
	\end{equation}

	\newcommand{\ufunction}{u}
	\newcommand{\rone}{r_{1}}
	\newcommand{\mintworone}{\min\set{2,\rone}}
	\newcommand{\competitor}{F}
	\newcommand{\dilation}{\lambda}
	\newcommand{\dimOmegadimOvers}{\frac{\dimension\omegadim}{s}}
	
	\textbf{STEP 1: proof of the inclusion $\minimizer\subset\BallRadius{1+\rzero}$.}
	
	For every $r>0$ we define the non-negative and monotone non-increasing absolutely continuous function
	\begin{equation*}
		\ufunction(r)\coloneqq|\minimizer\setminus\BallRadius{r}|.
	\end{equation*}
Thanks to the monotonicity and the hypotheses we observe that for every $r\geq 1$
\begin{equation}\label{eq:Loneestimate}
	\ufunction(r)\leq\ufunction(1)\leq\Cdimspotential\eps.
\end{equation}

	If the set $\set*{r> 1 : \ufunction(r)>0}$ is empty, there is nothing to prove, so we assume it is not. We let
	\begin{equation}\label{eq:defrone}
		\rone\coloneqq\sup\set*{r> 1 : \ufunction(r)>0}\in(1,+\infty].
	\end{equation}
	Our goal is to prove that
\begin{equation}\label{eq:GoalOuterEstimate}
	\rone-1\leq \Cdimspotential \eps^{s/\dimension}.
\end{equation}
	
	We fix $r>0$ such that $1<r<\mintworone$.
	By definition of $\rone$ we have that $|\minimizer\cap\BallRadius{r}|<|\minimizer|$.

	We consider the following competitor (with volume equal to $|\minimizer|=|\BallRadius{1}|$) for the volume constrained problem,
	\begin{equation*}
		\competitor\coloneqq\dilation(\minimizer\cap\BallRadius{r}),
	\end{equation*}
	with
	\begin{equation*}
		\dilation=\dilation(r)\coloneqq\tonde*{\frac{|\minimizer|}{|\minimizer\cap\BallRadius{r}|}}^{\frac{1}{\dimension}}>1.
	\end{equation*}
	After a computation, thanks to $\eqref{eq:Loneestimate}$, we observe that $\dilation$ is arbitrarily close to $1$ (upon choosing $\soglia$ small enough). In particular (upon choosing $\soglia$ small enough) we get
	\begin{equation}\label{eq:dilationEstimate}
		0<\dilation-1\leq\Cdim\ufunction(r)
	\end{equation}
	and, since $r<2$,
	\begin{equation*}
		\competitor\subseteq\BallRadius{\dilation r}\subseteq\BallRadius{3}.
	\end{equation*}
	
	At this point we make use of the minimality of $\minimizer$ (see \eqref{eq:RescaledSetSolvesRescPB}) to write
	\begin{align*}
		&\Ps(\minimizer) + \int_{\minimizer}\volScaledPotential(x)\integralde x\\
		&\leq \Ps(\competitor) + \int_{\competitor}\volScaledPotential(x)\integralde x=\dilation^{\dimension-s}\Ps(\minimizer\cap\BallRadius{r}) + \dilation^{\dimension}\int_{\minimizer\cap\BallRadius{r}}\volScaledPotential(\dilation x)\integralde x
	\end{align*}
	Now, we recall \cite[Lemma 2.1]{DNRV15} and making use of the isoperimetric inequality for $\Ps$  and the fact that $\minimizer=(\minimizer\setminus\BallRadius{r})\sqcup(\minimizer\cap\BallRadius{r})$ we obtain
	\begin{align}
		&\Cdims(\ufunction(r))^{\frac{\dimension-s}{\dimension}}\leq\Ps(\minimizer\setminus\BallRadius{r})=\Ps(\minimizer)-\Ps(\minimizer\cap\BallRadius{r})+2\integral{\minimizer\setminus\BallRadius{r}}\integral{\minimizer\cap\BallRadius{r}}\frac{1}{\abs{x-y}^{\dimension+s}}\integralde x\de y\nonumber\\
		&\leq (\dilation^{\dimension-s}-1)\Ps(\minimizer\cap\BallRadius{r})\label{eq:termFractionalPerimeter}\\
		&+\dilation^{\dimension}\int_{\minimizer\cap\BallRadius{r}}\volScaledPotential(\dilation x)\integralde x-\int_{\minimizer}\volScaledPotential(x)\integralde x \label{eq:termPotential}\\
		&+2\integral{\minimizer\setminus\BallRadius{r}}\integral{\minimizer\cap\BallRadius{r}}\frac{1}{\abs{x-y}^{\dimension+s}}\integralde x\de y.\label{eq:termDoubleIntegral}
	\end{align}
	We first estimate \eqref{eq:termFractionalPerimeter} using \eqref{eq:unifboundPSmin}, \eqref{eq:Loneestimate}, \eqref{eq:dilationEstimate} (possibly choosing a smaller $\soglia$) and the convexity of balls (see \cite[Lemma B.1]{FFMMM15}):
	\begin{align*}
		(\dilation^{\dimension-s}-1)\Ps(\minimizer\cap\BallRadius{r})\leq \Cdims(\dilation-1)\Ps(\minimizer)\leq \Cdimspotential\ufunction(r).
	\end{align*}

	We then estimate \eqref{eq:termPotential} using the fact that $\potential$ is locally Lipschitz (see \eqref{H:locallyLipschitz}) and non-negative (see \eqref{eq:Minofg}) together with \eqref{eq:Loneestimate}, \eqref{eq:dilationEstimate}, \eqref{eq:technicalSmallLength} and the fact that $r<2$ (possibly choosing a smaller $\soglia$):
	\begin{align*}
			&\dilation^{\dimension}\int_{\minimizer\cap\BallRadius{r}}\volScaledPotential(\dilation x)\integralde x-\int_{\minimizer}\volScaledPotential(x)\integralde x \\
			&=(\dilation^{\dimension}-1)\int_{\minimizer\cap\BallRadius{r}}\volScaledPotential(\dilation x)\integralde x
			+\int_{\minimizer\cap\BallRadius{r}}(\volScaledPotential(\dilation x)-\volScaledPotential(x))\integralde x -\int_{\minimizer\setminus\BallRadius{r}}\volScaledPotential(x)\integralde x \\
			&\leq\Cdim(\dilation-1)|\minimizer\cap\BallRadius{r}|\sup_{\BallRadius{3+\Rzero}}\potential
			+(\dilation-1)r\Lip(\potential;\BallRadius{3+\Rzero})|\minimizer\cap\BallRadius{r}|\\
			&\leq \Cdimspotential\ufunction(r).
	\end{align*}
	Now we deal with \eqref{eq:termDoubleIntegral}. Arguing exactly as in \cite[equation (15)]{CesNov17}, we find 
	
	\begin{align*}
		\integral{\minimizer\setminus\BallRadius{r}}\integral{\minimizer\cap\BallRadius{r}}\frac{1}{\abs{x-y}^{\dimension+s}}\integralde x\de y\leq-\dimOmegadimOvers\int_{r}^{+\infty}\frac{\ufunction'(t)}{(t-r)^s}\integralde t.
	\end{align*}
	Gathering all the previous inequalities together, we find that for every  $1<r<\mintworone$,
	\begin{align}\label{eq:INTDIFFone}
		\Cdims(\ufunction(r))^{\frac{\dimension-s}{\dimension}}\leq \Cdimspotential\ufunction(r)-\dimOmegadimOvers\int_{r}^{+\infty}\frac{\ufunction'(t)}{(t-r)^s}\integralde t.
	\end{align}
	Now, recalling \eqref{eq:Loneestimate}, upon choosing $\soglia$ smaller, we can rewrite \eqref{eq:INTDIFFone} as
	\begin{align}\label{eq:INTDIFFtwo}
		\Cdimspotential(\ufunction(r))^{\frac{\dimension-s}{\dimension}}\leq -\dimOmegadimOvers\int_{r}^{+\infty}\frac{\ufunction'(t)}{(t-r)^s}\integralde t.
	\end{align}
	
	 Now we integrate both sides of \eqref{eq:INTDIFFtwo} with respect to the variable $r$ and apply \cref{lem:estimateForFracDerivative} part $\ref{item:decreasing_case}$ to obtain that for some $\CdimspotentialUNCHG=\CdimspotentialUNCHG(\dimension,s,\potential)>0$	 
	 \begin{align}\label{eq:INTDIFFthree}
	 	2\CdimspotentialUNCHG\int_{\rho}^{\mintworone}\ufunction(r)^{\frac{\dimension-s}{\dimension}}\integralde r\leq \ufunction(\rho),
	 \end{align}
 	for every $1<\rho<\mintworone$.
 	
 	Recalling \eqref{eq:Loneestimate}, upon choosing $\Cdimspotential$ larger and $\soglia$ smaller, we can assume
 	\begin{align}\label{eq:nottofar}
 		1+\frac{(2\ufunction(1))^{\frac{s}{\dimension}}\dimension}{s\CdimspotentialUNCHG}\leq
 		1+\Cdimspotential\eps^{\frac{s}{\dimension}}<2.
 	\end{align}
 	Without loss of generality we can also assume that
 	\begin{equation}\label{eq:smallerthanrone}
 		1+\frac{(2\ufunction(1))^{\frac{s}{\dimension}}\dimension}{s\CdimspotentialUNCHG}<\rone,
 	\end{equation} otherwise \eqref{eq:nottofar} immediately implies \eqref{eq:GoalOuterEstimate}.

 	\eqref{eq:INTDIFFthree}, \eqref{eq:nottofar}, \eqref{eq:smallerthanrone} allow us to invoke \cref{lem:CompLikeDNRV} part $\ref{item:DNRVdecreasing_case}$ to deduce that
 	\begin{equation}\label{eq:ufunctionNull}
 		\ufunction\tonde*{1+\frac{(2\ufunction(1))^{\frac{s}{\dimension}}\dimension}{s\CdimspotentialUNCHG}}=0.
 	\end{equation} From \eqref{eq:ufunctionNull}, recalling \eqref{eq:nottofar} and \eqref{eq:defrone}, we obtain that $\rone\leq1+\frac{(2\ufunction(1))^{\frac{s}{\dimension}}\dimension}{s\CdimspotentialUNCHG}\leq 1+\Cdimspotential\eps^{\frac{s}{\dimension}}$.
 	
\newcommand{\wfunction}{w}
\newcommand{\rtwo}{r_{2}}
\newcommand{\competitorTwo}{G}
\newcommand{\dilationTwo}{\mu}
\textbf{STEP 2: proof of the inclusion $\BallRadius{1-\rzero}\subset\minimizer$.}

Thanks to the previous step, upon choosing $\volumezero=\volumezero(\dimension, s,\potential)$, we can assume 
	\begin{equation}\label{eq:technicalMinimizerUnifBounded}
		\minimizer\subset\BallRadius{2}.
	\end{equation}
For every $r>0$ we define the non-negative and monotone non-decreasing absolutely continuous function
\begin{equation*}
	\wfunction(r)\coloneqq|\BallRadius{r}\setminus\minimizer|.
\end{equation*}
Thanks to the monotonicity and the hypotheses we observe that for every $0\leq r\leq 1$
\begin{equation}\label{eq:LoneestimateTwo}
	\wfunction(r)\leq\wfunction(1)\leq\Cdimspotential\eps.
\end{equation}

We observe that $0\in\set*{r\in[0,1) : \wfunction(r)=0}$ and we let
\begin{equation}\label{eq:defrtwo}
	\rtwo\coloneqq\sup\set*{r\in[0,1) : \wfunction(r)=0}\in[0,1].
\end{equation}
Our goal is to prove that
\begin{equation}\label{eq:GoalInnerEstimate}
	1-\rtwo\leq \Cdimspotential \eps^{s/\dimension}.
\end{equation}

Without loss of generality we suppose that $\rtwo<1$, otherwise \eqref{eq:GoalInnerEstimate} trivialises itself, and we fix $r>0$ such that $\rtwo<r<1$.
By definition of $\rtwo$ we have that $|\minimizer\cup\BallRadius{r}|>|\minimizer|$.

We consider the following competitor (with volume equal to $|\minimizer|=|\BallRadius{1}|$) for the volume constrained problem,
\begin{equation*}
	\competitorTwo\coloneqq\dilationTwo(\minimizer\cup\BallRadius{r}),
\end{equation*}
with
\begin{equation*}
	\dilationTwo=\dilationTwo(r)\coloneqq\tonde*{\frac{|\minimizer|}{|\minimizer\cup\BallRadius{r}|}}^{\frac{1}{\dimension}}<1.
\end{equation*}
After a computation, thanks to $\eqref{eq:LoneestimateTwo}$, we observe that $\dilationTwo$ is arbitrarily close to $1$ (upon choosing $\soglia$ small enough). In particular (upon choosing $\soglia$ small enough) we get
\begin{equation}\label{eq:dilationEstimateTwo}
	0<1-\dilationTwo\leq\Cdim\wfunction(r)
\end{equation}
and, since $r<1$,
\begin{equation*}
	\competitor\subseteq\BallRadius{\dilationTwo r}\subseteq\BallRadius{1}.
\end{equation*}
At this point we make use of the minimality of $\minimizer$ (see \eqref{eq:RescaledSetSolvesRescPB}) to write
\begin{align*}
	&\Ps(\minimizer) + \int_{\minimizer}\volScaledPotential(x)\integralde x\\
	&\leq \Ps(\competitorTwo) + \int_{\competitorTwo}\volScaledPotential(x)\integralde x=\dilationTwo^{\dimension-s}\Ps(\minimizer\cup\BallRadius{r}) + \dilationTwo^{\dimension}\int_{\minimizer\cup\BallRadius{r}}\volScaledPotential(\dilationTwo x)\integralde x
\end{align*}
Now, we recall \cite[Lemma 2.1]{DNRV15} and making use of the isoperimetric inequality for $\Ps$, the invariance under complementation of $\Ps$ and the fact that $\comp{\minimizer}=(\BallRadius{r}\setminus\minimizer)\sqcup\comp{(\BallRadius{r}\cup\minimizer)}$ we obtain
\begin{align}
	&\Cdims(\wfunction(r))^{\frac{\dimension-s}{\dimension}}
	\leq\Ps(\BallRadius{r}\setminus\minimizer) \nonumber\\
	&=\Ps(\comp{\minimizer})-\Ps(\comp{(\minimizer\cup\BallRadius{r})})+2\integral{\BallRadius{r}\setminus\minimizer}\integral{\comp{(\minimizer\cup\BallRadius{r})}}\frac{1}{\abs{x-y}^{\dimension+s}}\integralde x\de y\nonumber\\
	&\leq(\dilationTwo^{\dimension-s}-1)\Ps(\minimizer\cup\BallRadius{r})\label{eq:termFractionalPerimeterTwo}\\
	&+\dilationTwo^{\dimension}\int_{\minimizer\cup\BallRadius{r}}\volScaledPotential(\dilationTwo x)\integralde x-\int_{\minimizer}\volScaledPotential(x)\integralde x \label{eq:termPotentialTwo}\\
	&+2\integral{\BallRadius{r}\setminus\minimizer}\integral{\comp{(\minimizer\cup\BallRadius{r})}}\frac{1}{\abs{x-y}^{\dimension+s}}\integralde x\de y.\label{eq:termDoubleIntegralTwo}
\end{align}

We first observe that \eqref{eq:termFractionalPerimeterTwo} is negative.

We then estimate \eqref{eq:termPotentialTwo} using the fact that $\potential$ is locally Lipschitz (see \eqref{H:locallyLipschitz}) and non-negative (see \eqref{eq:Minofg}) together with \eqref{eq:technicalMinimizerUnifBounded}, \eqref{eq:dilationEstimateTwo}, \eqref{eq:technicalSmallLength} and the fact that $r<2$:
	\begin{align*}
			&\dilationTwo^{\dimension}\int_{\minimizer\cup\BallRadius{r}}\volScaledPotential(\dilationTwo x)\integralde x-\int_{\minimizer}\volScaledPotential(x)\integralde x \\
			&=(\dilationTwo^{\dimension}-1)\int_{\minimizer\cup\BallRadius{r}}\volScaledPotential(\dilationTwo x)\integralde x+\int_{\minimizer}(\volScaledPotential(\dilationTwo x)-\volScaledPotential(x))\integralde x +\int_{\BallRadius{r}\setminus\minimizer}\volScaledPotential(\dilationTwo x)\integralde x \\
			&\leq (1-\dilationTwo)2\Lip(\potential;\BallRadius{2+\Rzero}) + \wfunction(r) \sup_{\BallRadius{1+\Rzero}}\potential\\
			&\leq \Cdimspotential\wfunction(r).
	\end{align*}
	
	Now we deal with \eqref{eq:termDoubleIntegralTwo}. Following th scheme of the proof of \cite[equation (15)]{CesNov17}, we write
	
	\begin{align*}
		&\integral{\BallRadius{r}\setminus\minimizer}\integral{\comp{(\minimizer\cup\BallRadius{r})}}\frac{1}{\abs{x-y}^{\dimension+s}}\integralde x\de y
		\leq\integral{\BallRadius{r}\setminus\minimizer}\integral{\comp{\BallRadius{r}}}\frac{1}{\abs{x-y}^{\dimension+s}}\integralde x\de y\\
		&\leq\integral{\BallRadius{r}\setminus\minimizer}\integral{\Rdim\setminus\BallRadiusCenter{r-|y|}{y}}\frac{1}{\abs{x-y}^{\dimension+s}}\integralde x\de y
	    \leq\integral{\BallRadius{r}\setminus\minimizer}\integral{\Rdim\setminus\BallRadius{r-|y|}}\frac{1}{\abs{z}^{\dimension+s}}\integralde z\de y\\
	    &\leq\dimension\omegadim\integral{\BallRadius{r}\setminus\minimizer}\int_{r-|y|}^{+\infty}\frac{1}{\rho^{1+s}}\integralde \rho\de y=\frac{\dimension\omegadim}{s}\integral{\BallRadius{r}\setminus\minimizer}\frac{1}{(r-|y|)^s}\integralde y\\
	    &=\frac{\dimension\omegadim}{s}\int_{0}^{r}\frac{1}{(r-t)^s}\HdimMinusOne(\partial\BallRadius{t}\setminus\minimizer)\integralde t\\
		&=\dimOmegadimOvers\int_{0}^{r}\frac{\wfunction'(t)}{(r-t)^s}\integralde t.
	\end{align*}
	
	Gathering all the previous inequalities together, we find that for every $r$ s.t. $\rtwo<r<1$,
		\begin{align}\label{eq:INTDIFFoneSTEPTWO}
			\Cdims(\wfunction(r))^{\frac{\dimension-s}{\dimension}}\leq \Cdimspotential\wfunction(r)+\dimOmegadimOvers\int_{0}^{r}\frac{\wfunction'(t)}{(r-t)^s}\integralde t.
		\end{align}
		Now, recalling \eqref{eq:LoneestimateTwo}, upon choosing $\soglia$ smaller, we can rewrite \eqref{eq:INTDIFFoneSTEPTWO} as
		\begin{align}\label{eq:INTDIFFtwoSTEPTWO}
			\Cdimspotential(\wfunction(r))^{\frac{\dimension-s}{\dimension}}\leq \dimOmegadimOvers\int_{0}^{r}\frac{\wfunction'(t)}{(r-t)^s}\integralde t.
		\end{align}
		
		 Now we integrate both sides of \eqref{eq:INTDIFFtwoSTEPTWO} with respect to the variable $r$ and apply \cref{lem:estimateForFracDerivative} part $\ref{item:increasing_case}$ to obtain that for some $\CdimspotentialUNCHG=\CdimspotentialUNCHG(\dimension,s,\potential)>0$	 
		 \begin{align}\label{eq:INTDIFFthreeSTEPTWO}
		 	2\CdimspotentialUNCHG\int_{\rtwo}^{\rho}\wfunction(r)^{\frac{\dimension-s}{\dimension}}\integralde r\leq \wfunction(\rho),
		 \end{align}
	 	for every $\rtwo<\rho<1$.
	 	
	 	Recalling \eqref{eq:LoneestimateTwo}, upon choosing $\Cdimspotential$ larger and $\soglia$ smaller, we can assume
	 	\begin{align}\label{eq:nottofarSTEPTWO}
	 		1-\frac{(2\wfunction(1))^{\frac{s}{\dimension}}\dimension}{s\CdimspotentialUNCHG}\geq
	 		1-\Cdimspotential\eps^{\frac{s}{\dimension}}>0.
	 	\end{align}
	 	Without loss of generality we can also assume that
	 	\begin{equation}\label{eq:smallerthanroneSTEPTWO}
	 		1-\frac{(2\wfunction(1))^{\frac{s}{\dimension}}\dimension}{s\CdimspotentialUNCHG}>\rtwo,
	 	\end{equation} otherwise \eqref{eq:nottofarSTEPTWO} immediately implies \eqref{eq:GoalInnerEstimate}.

	 	\eqref{eq:INTDIFFthreeSTEPTWO} and \eqref{eq:smallerthanroneSTEPTWO} allow us to invoke \cref{lem:CompLikeDNRV} part $\ref{item:DNRVincreasing_case}$ to deduce that
	 	\begin{equation}\label{eq:wfunctionNull}
	 		\wfunction\tonde*{1-\frac{(2\wfunction(1))^{\frac{s}{\dimension}}\dimension}{s\CdimspotentialUNCHG}}=0.
	 	\end{equation} From \eqref{eq:wfunctionNull}, recalling \eqref{eq:nottofarSTEPTWO} and \eqref{eq:defrtwo}, we obtain that $\rtwo\geq1-\frac{(2\wfunction(1))^{\frac{s}{\dimension}}\dimension}{s\CdimspotentialUNCHG}\geq 1-\Cdimspotential\eps^{\frac{s}{\dimension}}$.
\end{proof}
\end{@empty}

\section{Unconstrained Minimization Problem}\label{sec:UncoMinProblem}
Now we show that our minimization problem \eqref{MinimizationNonlocalProblem}
is equivalent to the rescaled and unconstrained problem 
\begin{equation*}
	\min\left\{ \EucEnergysScaledPotentialVolume{\mu}(E) \coloneqq \Ps(E) + \int_{E}\volScaledPotential(x)\integralde x + \mu ||E| - |B_1|| \SetSep E\subset \Rdim \right\}
\end{equation*}
for $\mu>0$ sufficiently large. Here we recall that $\volScaledPotential$ is defined by \eqref{eq:Def_VolScaledPotential}, that is, 
\begin{equation*}
	\volScaledPotential(x) \coloneqq \length^s\,\potential(\length\,x + \xvolume)
\end{equation*}
for $x \in \Rdim$ where $\length$ is defined by \eqref{eq:defOfLength} and $\xvolume$ is given in \cref{prop:properties_of_minimizers}.

\begin{proposition}[Uniform version of Lemma 3.4 in \cite{CesNov17}]\label{prop:Solutions_of_Penalized_Problem(Uniform_version)}
		\newcommand{\penalZero}{\mu_{0}}
		\newcommand{\Rraggio}{3}	
		Let $s\in(0,1)$ and assume that $\potential$ satisfies \ref{H:coercive}, \ref{H:locallyLipschitz} and \eqref{eq:Minofg}. For every $\volume>0$ let $\Evol$ be a minimizer of $\EucEnergysPotential$ with $|\Evol|=\volume$ and let $\EtildeVol$ be as in \eqref{eq:Def_Rescaled_Minimizer}.
		There exist $\volumezero=\volumezero(\dimension, s,\potential)>0$ and $\penalZero=\penalZero(\dimension,s,\potential)>0$ such that if $0<\volume\leq\volumezero(\dimension, s,\potential)$, then,
		\begin{equation*}
			\EtildeVol\subset\BallRadius{\Rraggio/2},
		\end{equation*}
		and 
		\begin{equation}\label{eq:RescaledSetSolvesRescPBWithPenaliz}
				\EtildeVol\in\argmin\left\{ \Ps(F) + \int_{F}\volScaledPotential(x)\integralde x +\penal\abs*{|F|-|\Bone|}\SetSep  F\subset\BallRadius{\Rraggio}\right\},
			\end{equation}
		for every $\penal\geq\penalZero$, where $\volScaledPotential$ is defined by \eqref{eq:Def_VolScaledPotential}.
\end{proposition}
\begin{proof}

		\newcommand{\nk}{{n_k}}
		
		\newcommand{\Gtildepen}{\widetilde{G}_{\penal}}
		\newcommand{\lambdapen}{\lambda_{\penal}}
		\newcommand{\lipnk}{\varphi_\nk} 
		\newcommand{\intSetFunctionDex}[2]{\int_{#1}{#2}\thinspace\mathrm{d}x}
		\newcommand{\di}{\mathop{}\!\mathrm{d}}
		\newcommand{\Rraggio}{3}
		\newcommand{\RraggioDoppio}{6}
		\newcommand{\penalZero}{\mu_{0}}
		\newcommand{\V}[1]{\integral{#1}\volScaledPotential(x)\integralde x}
		\newcommand{\Gpen}{G_{\penal}}
		\newcommand{\clocale}{c}
		\newcommand{\clocaleEsplicita}{\Ps(\BallRadius{1})+(\norm{\potential}_{\Lspace{\infty}(\BallRadius{\Rraggio+\Rzero})})2|\BallRadius{\Rraggio}|}
		\newcommand{\clocaleOne}{{\varepsilon}_{1}}
		\newcommand{\clocaleTwo}{\varepsilon_{2}}
		\newcommand{\varmuta}{\lambda}
		\newcommand{\cBrutta}{c'}
		\newcommand{\cBruttaEsplicita}{\quadre*{2\clocale\frac{(\dimension-s)}{\dimension\omegadim}+\Lip(\potential;\BallRadius{\RraggioDoppio+\Rzero})\frac{2\diam(\BallRadius{\Rraggio})}{\dimension}}(1+\clocaleOne)^\dimension+\norm{\potential}_{\Lspace{\infty}(\BallRadius{\Rraggio+\Rzero})}}
		
		We revisit the proof of \cite[Lemma 3.4]{CesNov17} making use of the extra information given by \cref{thm:Uniform_closeness_of_minimizers_to_ball} and tracking carefully all the constants.
		
		In this proof $\clocale=\clocale(\dimension, s,\potential)>0, \cBrutta=\cBrutta(\dimension, s,\potential)>0$ are fixed constants defined by,
		\begin{align*}
		&\clocale=\clocaleEsplicita,\\
		&\cBrutta=\cBruttaEsplicita,
		\end{align*}
		where $\Rzero=\Rzero(\dimension, s,\potential)$ is as in the statement of \cref{prop:properties_of_minimizers}.
		
		Let $\clocaleOne=\clocaleOne(\dimension, s)\in(0,1/2)$ be a small constant such that
		\begin{align}\label{eq:technical_elementary_inequality}
		&\abs*{\frac{\varmuta-1}{\varmuta^\dimension-1}}\leq\frac{2}{\dimension} &\abs*{\frac{\varmuta^{\dimension-s}-1}{\varmuta^\dimension-1}}\leq\frac{2(\dimension-s)}{\dimension}
		\end{align}
		for all $\varmuta>0$ with $|\varmuta-1|\leq\clocaleOne$.
		
		Let $\clocaleTwo=\clocaleTwo(\dimension, s)>0$ be a constant such that $\clocaleTwo<\min \graffe*{1-\frac{1}{(1+\clocaleOne)^{\dimension}},\frac{1}{(1-\clocaleOne)^{\dimension}}-1}$. 
		
		Let $\penalZero=\penalZero(\dimension,s,\potential)>0$ be a constant such that $\penalZero>\max\graffe*{\frac{\clocale}{\clocaleTwo|\BallRadius{1}|},\cBrutta}$.

		We fix $\volumezero=\volumezero(\dimension, s,\potential)\in(0,1)$ as in the statement of \cref{thm:Uniform_closeness_of_minimizers_to_ball} (or smaller)  and $\volume$ such that $0<\volume\leq\volumezero(\dimension, s,\potential)$. 
		
		We fix  $\penal\geq\penalZero$ and we assume by contradiction that  $\exists \Gpen\subseteq\BallRadius{\Rraggio}$ of finite $s$-perimeter such that
			\begin{align}\label{key'}
			&\Ps(\Gpen)+\V{\Gpen}+\penal\Big| |\Gpen|-|\BallRadius{1}|\Big|<\Ps(\EtildeVol)+\V{\EtildeVol}
			\end{align}
		Then, since $\EtildeVol$ solves \eqref{eq:RescaledSetSolvesRescPB},
			\begin{align*}
			&\Ps(\Gpen)+\penal\Big| |\Gpen|-|\BallRadius{1}|\Big|
			\leq\Ps(\EtildeVol)+\V{\EtildeVol}-\V{\Gpen}\leq\\
			&\leq\Ps(\BallRadius{1})+\V{\BallRadius{1}}-\V{\Gpen}\leq\\
			&\leq \Ps(\BallRadius{1})+(\norm{\volScaledPotential}_{\Lspace{\infty}(\BallRadius{\Rraggio})})|\Gpen\symmDiff \BallRadius{1}|\leq \clocaleEsplicita=\clocale.
			\end{align*}
			
			Which, implies
			\begin{align*}
			&	\Ps(\Gpen)\leq\clocale & \abs*{|\Gpen|-|\BallRadius{1}|}\leq\frac{\clocale}{\penal}\leq\frac{\clocale}{\penalZero}\leq\clocaleTwo|\BallRadius{1}|.
			\end{align*}
			In particular, $\frac{|\BallRadius{1}|}{(1+\clocaleOne)^\dimension}\leq|\Gpen|\leq\frac{|\BallRadius{1}|}{(1-\clocaleOne)^\dimension}$. We observe that $|\Gpen|\neq|\BallRadius{1}|$, since $\EtildeVol$ solves \eqref{eq:RescaledSetSolvesRescPB} and \eqref{key'} holds.

			At this point let $\Gtildepen:=\lambdapen \Gpen$, with $\lambdapen:=(|\BallRadius{1}|/|\Gpen|)^{1/\dimension}$. Being the volume of $\Gtildepen$ equal to $|\BallRadius{1}|$, by minimality of $\EtildeVol$ for \eqref{eq:RescaledSetSolvesRescPB}, we get:
			\begin{align}\label{key2}
			\Ps(\EtildeVol)+\V{\EtildeVol}\leq \Ps(\Gtildepen)+\V{\Gtildepen}.
			\end{align}

			The combination between \eqref{key'} and \eqref{key2} leads to:
			\begin{align}\label{key3}
			\penal\leq\frac{\Ps(\Gtildepen)-\Ps(\Gpen)}{\left||\Gpen|-|\BallRadius{1}|\right|}+\frac{\V{\Gtildepen}-\V{\Gpen}}{\left||\Gpen|-|\BallRadius{1}|\right|}
			\end{align}
	
			Applying exactly the same argument in the proof of \cite[Lemma 3.4]{CesNov17}, we obtain:

		\begin{equation}\label{conclude1}
		\frac{\Ps(\Gtildepen)-\Ps(\Gpen)}{\left||\Gpen|-|\BallRadius{1}|\right|}=\frac{(\lambdapen^{\dimension-s}-1)\Ps(\Gpen)}{|\lambdapen^{\dimension}-1||\Gpen|},
		\end{equation}
		and
		\begin{align}\label{conclude1bis}
		&\nonumber\frac{\V{\Gtildepen}-\V{\Gpen}}{\left||\Gpen|-|\BallRadius{1}|\right|}\\
		&=\frac{\V{\Gtildepen}-\lambdapen^\dimension\V{\Gpen}}{\left||\Gpen|-|\BallRadius{1}|\right|}+\frac{(\lambdapen^\dimension-1)\V{\Gpen}}{\left||\Gpen|-|\BallRadius{1}|\right|}\\
		&\nonumber\leq \frac{|\BallRadius{1}|\Lip(\potential;\BallRadius{\RraggioDoppio+\Rzero})|\lambdapen-1|\diam(\BallRadius{\Rraggio})}{|\lambdapen^{\dimension}-1||\Gpen|}+\norm{\potential}_{\Lspace{\infty}(\BallRadius{\Rraggio+\Rzero})}.
		\end{align}

			By \eqref{eq:technical_elementary_inequality}, \eqref{key3}, \eqref{conclude1}, \eqref{conclude1bis} we deduce that

			\begin{align*}
			\penal\leq\cBrutta<\penalZero,
			\end{align*}
			 that is the desired contradiction.

\end{proof}

\section{Regularity and Convexity of Minimizers}\label{sec:Regular_Convex_Minimizers} \label{sectionRegularityConvexityMinimizers}
In this section, our goal is to show the following lemma:
\begin{lemma}\label{lemmaRegularityConvexityMinimizers}
	Let $s\in(0,1)$ and assume that $\potential$ satisfies \ref{H:coercive}, \ref{H:locallyLipschitz} and \eqref{eq:Minofg}. For every $\volume>0$ let $\Evol$ be a minimizer of $\EucEnergysPotential$ with $|\Evol|=\volume$ and let $\EtildeVol$ be as in \eqref{eq:Def_Rescaled_Minimizer}. Then, there exists $\widetilde{m}_0>0$, which depends only on $\dimension$, $s$ and $\potential$, such that, for any $\volume \in (0,\,\widetilde{m}_0)$ the boundary of $\EtildeVol$ is of class $C^{2,\alpha}$ for any $\alpha \in (0,\,s)$. Moreover, $\partial \EtildeVol$ converges to $\partial B_1$ in $C^{2}$-sense, in particular, $\EtildeVol$ is convex for any $\volume \in (0,\,\widetilde{m}_0)$.
\end{lemma}
We remark that Lemma \ref{lemmaRegularityConvexityMinimizers} actually implies that the latter part of Theorem \ref{theoremMainShapeMinimizers} is valid. Therefore, combining Lemma \ref{lemmaRegularityConvexityMinimizers} with Theorem \ref{thm:Uniform_closeness_of_minimizers_to_ball}, we conclude the proof of our main theorem, that is, Theorem \ref{theoremMainShapeMinimizers}.
 
In order to prove Lemma \ref{lemmaRegularityConvexityMinimizers}, we exploit the regularity results proved by Caputo and Guillen \cite{CapGui10}; Figalli, Fusco, Maggi, Millot, and Morini \cite{FFMMM15}; Savin and Valdinoci \cite{SavVal13}; and Barrios, Figalli, and Valdinoci \cite{BFV14}. Before recalling the regularity results, we first give the definition of the so called \enquote{almost} minimizers or $\Lambda$-minimizers for the $s$-fractional perimeter $\Ps$ based on the definition given by Figalli, Fusco, Maggi, Millot, and Morini \cite{FFMMM15} or Caputo and Guillen \cite{CapGui10}.
\begin{definition}[$(\Lambda, \delta)$-Minimizers or Almost Minimizers of $\Ps$] \label{definitionLocalALmostMinimizer}
	Let $s \in (0,\,1)$, $\Lambda>0$ and $\delta>0$. We say that a measurable set $E \subset \Rdim$ is a \textit{$(\Lambda,\delta)-$minimizer} of $\Ps$ if
	\begin{equation}\label{definitionAlmostMinimizer}
		\Ps(E;\BallRadiusCenter{r}{x}) \leq \Ps(F;\BallRadiusCenter{r}{x}) + \Lambda|E \symmDiff F|	
	\end{equation} 
	for any measurable set $F \subset \Rdim$, $x \in \partial E$, and $r\in(0,\delta)$ with $E \symmDiff F \subset B_r(x)$.
\end{definition}
Note that the similar concept of the almost minimality for $\Ps$, which is more general, was also given by Caputo and Guillen \cite{CapGui10}. See \cite{CapGui10} for the details.

Next we recall the regularity result on the almost minimizers or $(\Lambda, \delta)$-minimizers of $\Ps$ shown by Figalli, Fusco, Maggi, Millot, and Morini \cite[Corollary 3.5]{FFMMM15} (see also \cite{CapGui10}). This result
is a nonlocal analogue of the theory of Tamanini on almost minimal surfaces. 
\begin{theorem}[cf. Corollary 3.5 in \cite{FFMMM15}] \label{improvementFlatnessFFMMM15}
	If $\dimension \geq 2$, $\Lambda > 0$, and $s_0 \in (0,\,1)$, then there exist $0 < \varepsilon_0 < 1$, $C_0>0$, and $\alpha \in (0,\,1)$, depending on $n$, $\Lambda$, and $s_0$ only, with the following property: if $E$ is a $(\Lambda, \delta)-$minimizer of $\Ps$ for $s \in [s_0, \,1)$ in the sense of Definition \ref{definitionLocalALmostMinimizer} and 
	\begin{equation}\label{regularityCriterion}
		0 \in \partial E, \quad \partial E \cap B_1(0) \subset \{y \in \Rdim \SetSep |y \cdot e| < \varepsilon_0 \}
	\end{equation}
	for some $e \in \mathbb{S}^{\dimension-1}$, then $\partial E \cap B_{1/2}(0)$ is of class $C^{1,\alpha}$ and the $C^{1,\alpha}$-norm of its graph function is bounded by $C_0$.

\end{theorem}
In other words, \cref{improvementFlatnessFFMMM15} with \cite[Theorem 1.2]{CapGui10} implies that the boundary of an almost minimizer of $\Ps$ is of class $C^{1,\alpha}$ with some $\alpha \in (0,\,1)$ outside of a closed set of Hausdorff dimension at most $\dimension-2$.

Next we recall the regularity result of fractional minimal cones in $\R^2$ by Savin and Valdinoci \cite{SavVal13}.
\begin{theorem}[\cite{SavVal13}] \label{thmNonlocalCone2d}
	Assume that $E \subset \R^2$ is an $s$-fractional minimal cone, namely, $E$ is a minimizer of $P_s$ in any ball and satisfies that $E = t\,E$ for any $t >0$. Then $E$ is a half-plane.
\end{theorem}
In particular, by combining the blow-up and blow-down arguments in \cite{CRS10}, one may obtain that $s$-fractional minimal surfaces in $\R^2$ are fully $C^{1,\alpha}$-regular for any $\alpha \in (0,\,s)$.
\begin{corollary}[\cite{SavVal13}] \label{corNonlocalCone2d}
	If $E$ is an $s$-fractional minimal set in an open set $\Omega \subset \R^2$, then $\partial E \cap \Omega'$ is a $C^{1,\alpha}$-curve for any $\Omega' \Subset \Omega$.
\end{corollary}  

Originally, the regularity of nonlocal (fractional) minimal surfaces, which are defined by the boundaries of sets minimizing the fractional perimeter, was obtained by Caffarelli, Roquejoffre, and Savin \cite{CRS10}. Precisely, they proved that every fractional minimal surface is locally $C^{1,\alpha}$ outside of a closed singular set of Hausdorff dimension at most $\dimension-2$. Moreover, thanks to \cref{corNonlocalCone2d}, it follows that the singular set of fractional minimal surfaces has Hausdorff dimension at most $\dimension-3$.

Now we prove that the rescaled minimizers for our problem are $(\Lambda, \delta)-$minimizers of $\Ps$ where $\Lambda$ is independent of the minimizers.

\begin{lemma}[$(\Lambda, \delta)$-Minimality or Almost Minimality]\label{res:UniformAlmostMinimizer}
	Let $s \in (0,\,1)$ and $\delta\in(0,1)$. Assume that $\potential$ satisfies \ref{H:coercive}, \ref{H:locallyLipschitz} and \eqref{eq:Minofg}. For every $m>0$, let $\Evol$ be a minimizer of $\EucEnergysPotential$ with $|\Evol|=\volume$ and let $\EtildeVol$ be as in \eqref{eq:Def_Rescaled_Minimizer}. Then there exist $m_0=m_0(N,s,g)>0$ and $\Lambda=\Lambda(\dimension,s,\potential)$ such that, for every $m \in (0,\,m_0)$, $\EtildeVol$ is a $(\Lambda, \delta)-$minimizer of $\Ps$ in the sense of Definition \ref{definitionLocalALmostMinimizer}.
\end{lemma}

\begin{proof}
We fix $\delta\in(0,1)$.
From \cref{prop:Solutions_of_Penalized_Problem(Uniform_version)}, we know that $\EtildeVol$ is also a solution to the problem \eqref{eq:RescaledSetSolvesRescPBWithPenaliz},
for some large constant $\mu>0$ independent of $\volume$ and $\EtildeVol\subset\BallRadius{3/2}$. Now, letting $F \subset \Rdim$,  $x \in \partial \EtildeVol$, and $r\in(0,\delta)$ with $\EtildeVol \symmDiff F \subset B_r(x)$, from the minimality of $\EtildeVol$, we have that 
\begin{equation}\label{minimalitySublevel}
	\EucEnergysScaledPotentialVolume{\mu}(\EtildeVol) \leq \EucEnergysScaledPotentialVolume{\mu}(F),
\end{equation}
since $F\subset\BallRadius{3}$.
Hence, from \eqref{minimalitySublevel}, we can compute as follows:
\begin{align}\label{estimateQuasiNonlocalMini}
		&\Ps(\EtildeVol) - \Ps(F)\nonumber\\
		& = \EucEnergysScaledPotentialVolume{\mu}(\EtildeVol) - \int_{\EtildeVol}\volScaledPotential(x)\integralde x - \mu\abs*{|\EtildeVol|-|\Bone|} - \EucEnergysScaledPotentialVolume{\mu}(F) +\int_{F} \volScaledPotential(x)\integralde x + \mu\abs*{|F|-|\Bone|}  \nonumber\\
		&\leq \int_{\Rdim}|\chi_{\EtildeVol}- \chi_{F}|\,\volScaledPotential(x)\integralde x + \mu ||F|-|\EtildeVol|| \nonumber\\
		&\leq \int_{\EtildeVol \symmDiff F}\volScaledPotential(x)\integralde x + \mu |\EtildeVol \symmDiff F|. 
\end{align}
Moreover, from the fact that $\EtildeVol\subset\BallRadius{3/2}$ we deduce that $\EtildeVol \symmDiff F \subset B_r(x) \subset \BallRadius{3}$.

Therefore, by \cref{prop:properties_of_minimizers} and the definition of $\volScaledPotential$, we have
\begin{equation}\label{estimateResidue01}
		\int_{\EtildeVol \symmDiff F} \volScaledPotential(x)\integralde x  \leq \length^s\|\potential\|_{L^{\infty}(B_R(0))} \,|\EtildeVol \symmDiff F|,
\end{equation} 
where $R>0$ is a constant independent of $m$. Hence, from \eqref{estimateQuasiNonlocalMini} and \eqref{estimateResidue01}, we obtain
\begin{equation}\nonumber
	\Ps(\EtildeVol) \leq  \Ps(F) + (\sigma(m_0)^s\|\potential\|_{L^{\infty}(B_R(0))} + \mu) \,|\EtildeVol \symmDiff F|
\end{equation}
for any $F \subset \Rdim$, $x \in \partial \EtildeVol$, and $r\in(0,\delta)$ with $\EtildeVol \symmDiff F \subset B_r(x)$ where $m_0>0$ is as in Proposition \ref{prop:Solutions_of_Penalized_Problem(Uniform_version)}. This implies that $\EtildeVol$ is an almost minimizer or $(\Lambda, \delta)$-minimizer of $\Ps$ in the sense of \eqref{definitionAlmostMinimizer} with $\Lambda \coloneqq \sigma(m_0)^s\|\potential\|_{L^{\infty}(B_R(0))} + \mu < \infty$, which is independent of minimizers.
\end{proof}

As a consequence of \cref{res:UniformAlmostMinimizer} and the aforementioned regularity results for the almost minimizers of the $s$-fractional perimeter, we obtain 
\begin{lemma}[$C^{1,\alpha}$-regularity for Minimizers of $\EucEnergysVolScaledPotential$]\label{lemmaHolderRegularityMinimizers}
	Let $s \in (0,\,1)$ and assume that $\potential$ satisfies \ref{H:coercive}, \ref{H:locallyLipschitz} and \eqref{eq:Minofg}. For every $\volume>0$ let $\Evol$ be a minimizer of $\EucEnergysPotential$ with $|\Evol|=\volume$ and let $\EtildeVol$ be as in \eqref{eq:Def_Rescaled_Minimizer}. There exists $\volumezero=\volumezero(\dimension, s,\potential)>0$ s.t. if $\volume \in (0,\,\volumezero)$, then the boundary of $\EtildeVol$ is of class $C^{1,\alpha}$ for some $\alpha \in (0,\,1)$ outside of a closed set of Hausdorff dimension at most $\dimension-3$. 
\end{lemma}
\begin{proof}
 The idea is to apply \cref{improvementFlatnessFFMMM15} to our problem. To this aim, it is sufficient to show that every minimizer $\EtildeVol$ is indeed a $(\Lambda, \delta)-$minimizer in the sense of \eqref{definitionAlmostMinimizer} for some $\Lambda>0$ depending on $\dimension, s$ and $\potential$ only, which is guaranteed by \cref{res:UniformAlmostMinimizer}.

 Therefore, from \cref{improvementFlatnessFFMMM15} (see also \cite[Corollary 3.5]{FFMMM15}) and the regularity of minimal cones in \cite{SavVal13}, we obtain that the boundary of $\EtildeVol$ is of class $C^{1,\alpha}$ for some $\alpha \in (0,\,1)$ outside of a closed set of Hausdorff dimension at most $\dimension-3$.
\end{proof}

From the uniform closeness of minimizers to the Euclidean ball stated in \cref{thm:Uniform_closeness_of_minimizers_to_ball}, we can immediately obtain the Hausdorff convergence of the minimizers to the unit ball.

Then, with the Hausdorff convergence of the minimizers at our disposal, we can exploit the regularity criterion \eqref{regularityCriterion} and the smoothness of the limit set $\BallRadius{1}$ via the argument in \cite[Theorem 26.6]{Maggi12} in order to obtain the following result (see also \cite[Theorem 3.3 and Corollary 3.6]{FFMMM15}).
\begin{lemma}\label{lemmaConvergenceBoundaryMinimizers}
	Let $s \in (0,\,1)$. Assume that $\potential$ satisfies \ref{H:coercive}, \ref{H:locallyLipschitz} and \eqref{eq:Minofg}. For any $m>0$, let $\Evol$ be a minimizer of $\EucEnergysPotential$ with $|\Evol|=\volume$ and let $\EtildeVol$ be as in \eqref{eq:Def_Rescaled_Minimizer}. Then, there exist a constant $m_0>0$ and a bounded sequence of functions $\{\phi_m\}_{m \in (0,\,m_0)} \subset C^{1,\alpha} (\partial B_1)$ with some $\alpha \in (0,\,1)$ (independent of $m$) such that, for any $m \in (0,\,m_0)$,	\begin{equation}\label{representationBoundaryMinimizers}
		\partial \EtildeVol = \left\{ (1+\phi_m (x))\,x \SetSep x \in \partial B_1(0) \right\} \quad \text{and} \quad  \lim_{m \downarrow 0} \|\phi_m\|_{C^1(\partial B_1(0))}=0.
	\end{equation}
\end{lemma}   
\begin{proof}
\newcommand*{\zvolLimit}{z_{0}}
\newcommand*{\zvolLimitPrime}{\zvolLimit'}
\newcommand*{\zvolLimitN}{z_{0,\dimension}}
\newcommand*{\zvol}{z_{\volume}}
	First of all, we recall that, from \cref{thm:Uniform_closeness_of_minimizers_to_ball}, there exist positive constants $m_0=m_0(\dimension,s,\potential)$ and $C=C(\dimension,s,\potential)$ such that, if $m \in (0,\,m_0)$, $\EtildeVol$ satisfies the inclusion
	\begin{equation}\label{uniformClosenessMinimizersforRegularity}
		B_{1-r_0}(0) \subset \EtildeVol \subset B_{1+r_0}(0)
	\end{equation}
	for some $r_0 \in (0,\, Cm^{s^2/(2\dimension^2)})$. Thus, \eqref{uniformClosenessMinimizersforRegularity} implies that $\EtildeVol$ converges to $B_1(0)$ as $m \downarrow 0$ in $L^1$-sense and moreover, by definition, $\partial \EtildeVol$ converges to $\partial B_1(0)$ in the Hausdorff distance.
	
	Now, as we observe in \cref{res:UniformAlmostMinimizer} and \cref{lemmaHolderRegularityMinimizers}, we obtain that $\EtildeVol$ is also an almost minimizer of $\Ps$ and that $\partial \EtildeVol$ is of class $C^{1,\alpha}$ with some $\alpha \in (0,\,1)$ for $m \in (0,\,m_0)$ outside of a closed set of Hausdorff dimension at most $\dimension-3$. Although the singular sets can appear on each $\partial \widetilde{E}_m$, by combining the convergence of $\{\partial \EtildeVol\}_{m \in (0,\,m_0)}$ to $\partial B_1$ in the Hausdorff distance with the regularity criterion \eqref{regularityCriterion}, we can show that there exists a constant $m'_0>0$ such that, for each $m \in (0,\,m'_0)$, there exists a function $\phi_m \in C^{1,\alpha}(\partial B_1)$ such that
	\begin{equation*}
		\partial \EtildeVol = \left\{ (1+\phi_m(x))\,x \SetSep x \in \partial B_1 \right\}.
	\end{equation*}
	Therefore, the representation of $\partial \widetilde{E}_m$ in \eqref{representationBoundaryMinimizers} holds. We refer to \cite[Corollary 3.6]{FFMMM15} for the similar argument. 
	
	Finally, we show the convergence of the functions $\{\phi_m\}_{m \in (0,\,m'_0) }$ in $C^1$ sense as in \eqref{representationBoundaryMinimizers} by using the argument in \cite[Theorem 26.6]{Maggi12}. Let $\zvol \in \partial \EtildeVol$ be any point and let $R \in (0,\,1)$. By the Hausdorff convergence of $\{\partial \EtildeVol\}_{m  \in (0,\,m'_0)}$, we can choose a subsequence (denoted by the same indices) such that $\zvol \to \zvolLimit \in \partial B_1(0)$ as $m \downarrow 0$. We define the cylinder $C(x,R) \subset \Rdim$ centred at $x = (x',\,x_{\dimension}) \in \R^{\dimension-1} \times \R$ by
	\begin{equation*}
		\{(y',\,y_\dimension) \in \mathbb{R}^{\dimension-1} \times \mathbb{R} \SetSep |y'-x'| < R ,\quad |y_\dimension - x_\dimension| < R\}.
	\end{equation*}
	Then, from the regularity of each $\partial \EtildeVol$ and $\partial B_1(0)$, for each $m \in (0,\, m'_0)$, we can choose a $C^{1,\alpha}$-function $u_m :B'_R(\zvolLimitPrime) \to \mathbb{R}$ and a smooth function $u_0 : B'_R(\zvolLimitPrime) \to \mathbb{R}$ such that $\Lip (u_m) \leq 1$ for any $m$ and 
	\begin{align}
		&C(\zvolLimit,R) \cap \EtildeVol = \{(x',\,x_\dimension) \SetSep -R < x_\dimension < u_m(x')\} \nonumber\\
		&C(\zvolLimit, R) \cap B_1(0) = \{(x',\,x_\dimension) \SetSep -R < x_\dimension < u_0(x') \} \nonumber
	\end{align}
	where $\zvolLimit \coloneqq (\zvolLimitPrime,\,\zvolLimitN) \in \R^{\dimension-1} \times \R$ and $B'_R(x') \subset \mathbb{R}^{\dimension-1}$ is the $(\dimension-1)$-dimensional ball of radius $R$ centred at $x' \in \mathbb{R}^{\dimension-1}$. From the regularity criterion \eqref{regularityCriterion} and the same argument as in, for instance, \cite[Theorem 26.3]{Maggi12}, we further obtain the H\"older continuity of the function $\nabla' u_m$, that is,
	\begin{equation}\label{holderEstimateGraphFunctions}
		|\nabla' u_m (y') - \nabla' u_m(z') | \leq  C\,|y'-z'|^{\alpha} < +\infty
	\end{equation}
	for any $y',\,z' \in B'_R(\zvolLimitPrime)$ and some constant $C>0$ independent of $m$, where $\nabla'$ is the gradient in $\R^{\dimension-1}$ and $\alpha \in (0,\,1)$ is as in \cref{lemmaHolderRegularityMinimizers}. Moreover, thanks to the convergence $\EtildeVol \to B_1(0)$ in $L^1$-sense, we have
	\begin{equation*}
		\int_{B'_R(\zvolLimitPrime)} |u_m(z') - u_0(z')| \,dz' = |\EtildeVol \symmDiff E \cap C(\zvolLimit,R)| \xrightarrow[m \downarrow 0]{} 0.
	\end{equation*}  
	Thus, we obtain
	\begin{equation}\label{convergenceGradientGraphFunctions}
		\int_{B'_R(\zvolLimitPrime)} \xi(z')\,\nabla' u_m(z') \,dz' \xrightarrow[m \downarrow 0]{} \int_{B'_R(\zvolLimitPrime)} \xi(z')\,\nabla' u_0(z')\,dz'
	\end{equation}
	for any $\xi \in C^{\infty}_c(B'_R(\zvolLimit))$. From \eqref{holderEstimateGraphFunctions}, we have that $\{\nabla' u_m\}_{m \in (0,\,m'_0)}$ is equi-continuous and bounded in $C^0(B'_R)$, and thus, by the Ascoli-Arzel\'a theorem, we can extract a subsequence (denoted by the same indices) such that $\nabla' u_m \to v$ as $m \downarrow 0$ uniformly in $B'_R(\zvolLimitPrime)$. From \eqref{convergenceGradientGraphFunctions}, the limit $v$ coincides with $\nabla' u_0$. Therefore, we obtain that the graph functions $\{u_m\}_{m \in (0,\, m'_0)}$ of $\partial \EtildeVol$ converge to the graph function $u_0$ of $\partial B_1$ in $C^1$ sense. From the choice of $\phi_m$, this implies that $\|\phi_m\|_{C^1(\partial B_1)}$ converges to 0. Therefore, we conclude the proof of \eqref{representationBoundaryMinimizers}.
\end{proof}

Now we improve the regularity of the boundary of minimizers of $\EucEnergysVolScaledPotential$ by employing the regularity result for solutions to integro-differential equations via the bootstrap argument. This result was obtained by Barrios, Figalli, and Valdinoci \cite[Theorem 1.6]{BFV14}. They proved the following regularity theorem on the solutions to integro-differential equations. For simplicity, we do not describe all of the theorem here. See \cite[Theorem 1.6]{BFV14} for the full statement. 
\begin{theorem}\label{theoremBootstrapBFV}
	Let $s \in (0,\,1)$, $\beta \in (0,\,1]$, and $r>0$. Let $v \in L^{\infty}(\R^{\dimension-1})$ be a solution (in the viscosity sense) to the integro-differential equation
	\begin{equation}\nonumber
		\int_{\R^{\dimension-1}}\Lr(x',\,y')\left( v(x'+y') + v(x'-y') - 2v(x') \right) \,\de y' = F(x', v(x'))
	\end{equation}
	for any $x' \in B'_r(0) \subset \R^{\dimension-1}$ where $\Lr$ satisfies the following assumptions:
	\begin{enumerate}[label=(A\arabic*)]
		\item\label{item:A_one} There exist constants $a_0,\,r_0>0$ and $\eta \in (0,\frac{a_0}{4})$ such that 
		\begin{equation*}
			\frac{(1-s)(a_0-\eta)}{|y'|^{\dimension+s}} \leq \Lr(x',\, y') \leq \frac{(1-s)(a_0+\eta)}{|y'|^{\dimension+s}} 
		\end{equation*} 
		for any $x' \in B'_r(0)$ and $y' \in B'_{r_0}(0) \setminus \{0\}$;
		
		\item\label{item:A_two} $\Lr \in C^{0,\beta}(B'_r(0) \times (\mathbb{R}^{N-1}\setminus\{0\}))$ and there exists a constant $C_0>0$ such that
		\begin{equation*}
			\| \Lr(\cdot,\,y') \|_{C^{0,\beta}(B'_r)} \leq \frac{C_0}{|y'|^{\dimension+s}}
		\end{equation*}
		for any $y' \in \mathbb{R}^{N-1} \setminus \{0\}$;
	\end{enumerate}
	and $F \in C^{0,\beta}(B'_r(0) \times \R)$. Then, if $\eta$ is sufficiently small, then $v \in C^{1+s+\alpha}(B'_{r/2}(0))$ for any $\alpha < \beta$. 
	Moreover, we have the estimate
	\begin{equation*}
		\|v\|_{C^{1+s+\alpha}(B'_{r/2}(0))} \leq C \, \left( 1 + \|v\|_{L^{\infty}(\R^{\dimension-1})} + \|F\|_{L^{\infty}(B'_r(0) \times \R)}\right),
	\end{equation*} 
	where $C>0$ is a constant depending on $\dimension$, $s$, $C_0$, and $\|F\|_{C^{0,\beta}(B'_r(0) \times \R)}$.
\end{theorem}

Taking into account all of the above arguments, we can obtain that the boundary of minimizers of $\EucEnergysVolScaledPotential$ with the volume $|B_1|$ has $C^{2,\beta}$-regularity for any $\beta \in (0,\,s)$. Precisely, we prove
\begin{lemma}[Improved Regularity of Minimizers]\label{lemmaImprovedRegularity}

	Let $s \in (0,\,1)$ and $m_0>0$ be as in Lemma \ref{lemmaConvergenceBoundaryMinimizers}. Assume that $\potential$ satisfies \ref{H:coercive}, \ref{H:locallyLipschitz} and \eqref{eq:Minofg}. Then, if $\Evol$ is a minimizer of $\EucEnergysPotential$ with $|\Evol|=\volume$ for any $\volume \in (0,\,m_0)$, $\partial \widetilde{E}_m$ is of class $C^{2,\beta}$ for any $\beta \in (0,\,s)$, where $\EtildeVol$ is defined as in \eqref{eq:Def_Rescaled_Minimizer}.
\end{lemma}
\begin{proof}
	We may assume that $0 \in \partial \widetilde{E}_m$ by translation. By \cref{lemmaConvergenceBoundaryMinimizers}, we can represent $\partial \EtildeVol \cap (B'_r(0) \times (- r,\, r) ) \subset \R^{\dimension-1} \times \mathbb{R}$ for $r>0$ as a graph of a $C^{1,\alpha}$-function $u$ for some $\alpha \in (0,\,1)$ by choosing the proper coordinate. From the fact that $\EtildeVol$ is a minimizer of $\EucEnergysVolScaledPotential$ with volume constraint and by employing the computation shown in \cite{BFV14}, we may obtain the following Euler-Lagrange equation in the viscosity sense:
	\begin{align}
		&\int_{\R^{\dimension-1}} \Lr(x',\,y') (u(x'+y') + u(x'-y') - 2u(x'))\,\de y' \nonumber\\
		&\quad  = G(x',\,u(x')) + \widetilde{\lambda}_m - \volScaledPotential(x',\,u(x')) \quad \text{for $x' \in B'_{r'/2}(0) \subset \R^{\dimension-1}$} \nonumber
	\end{align} 
	for $0 < r' < r$, where $\Lr$ satisfies \ref{item:A_one} and \ref{item:A_two} in \cref{theoremBootstrapBFV}, $\spt \Lr(x', \cdot) \subset B'_{r'/2}(0)$ for any $x' \in B'_{r'/2}(0)$, $G$ is some smooth function (see \cite[Section 3]{BFV14} for the precise expression), and $\widetilde{\lambda}_m$ is a Lagrange multiplier. Then, since the potential $\potential$ is locally Lipschitz, we now apply \cref{theoremBootstrapBFV} several times, if necessary, to conclude that the regularity of $u$ can be improved up to $C^{2,\beta}$ with $\beta \in (0,\,s)$. From the compactness of $\partial \EtildeVol$ and by a simple covering argument, we obtain the $C^{2, \beta}$-regularity of $\partial \EtildeVol$ for any $\beta \in (0,\,s)$ with the following estimate
	\begin{equation*}
		\|u\|_{C^{2,\beta}(B'_{r'/4})} \leq C \left( 1+ \|u\|_{L^{\infty}(B'_r)} + |\widetilde{\lambda}_m| + \|\volScaledPotential\|_{L^{\infty}(B_{2R_0})} \right),
	\end{equation*} where $C>0$ is a constant depending only on $\dimension$, $s$, and $\|\potential\|_{C^{0,1}(B_{2R_0})}$ and $R_0$ is as in Proposition \ref{prop:properties_of_minimizers}.

\end{proof}
Finally, we are ready to prove \cref{lemmaRegularityConvexityMinimizers}, that is, the main lemma in this section.
\begin{proof}[Proof of \cref{lemmaRegularityConvexityMinimizers}]
	Taking into account all the arguments in \cref{prop:properties_of_minimizers}, \cref{thm:Uniform_closeness_of_minimizers_to_ball}, \cref{lemmaConvergenceBoundaryMinimizers}, and \cref{lemmaImprovedRegularity}, we obtain that there exist a constant $m_0 = m_0(\dimension,s,\potential)>0$ and a sequence of functions $\{\phi_m\}_{m \in (0, \,m_0)} \subset C^{2,\alpha}(\partial B_1(0))$ for any $\alpha \in (0,\,s)$ such that
	\begin{align}
		&\partial \EtildeVol = \{ (1 + \phi_m(x))\,x \SetSep x \in \partial B_1 \}, \, \|\phi_m\|_{L^{\infty}(\partial B_1)} \leq C_1\, m^{\frac{s^2}{2\dimension^2}}, \, \lim_{m \downarrow 0}\|\phi_m\|_{C^1(\partial B_1)} = 0 \label{representationMinimizersNearBall}
	\end{align}
	where $C_1>0$ is a constant depending only on $\dimension$, $s$, and $\potential$. Moreover, from the proof of \cref{lemmaImprovedRegularity} and \eqref{representationMinimizersNearBall}, we can obtain
	\begin{equation}\label{representationMinimizerNearBall02}
		\|\phi_m\|_{C^{2,\alpha}(\partial B_1)} \leq C_2( 1 + |\widetilde{\lambda}_m|)
	\end{equation} 
	for any $m \in (0,\,m_0)$, by choosing $m_0>0$ small if necessary, where $C_2>0$ is a constant depending only on $\dimension$, $s$, and $\potential$ and $\widetilde{\lambda}_m$ is a Lagrange multiplier as in \cref{lemmaImprovedRegularity}. 
	
	Now we derive the upper bound of $|\widetilde{\lambda}_m|$ from the minimality of $\widetilde{E}_m$. Since $\EtildeVol$ is a minimizer of $\EucEnergysVolScaledPotential$ with $|\EtildeVol| = |B_1|$, by considering the Euler-Lagrange equation with volume constraint, we have
	\begin{equation}\label{eulerLagrangeEnergy}
		\frac{d}{dt}\Big\lfloor_{t=0}\left( \Ps(\Phi_t(\EtildeVol)) + \int_{\Phi_t(\EtildeVol)}\volScaledPotential \right) = \widetilde{\lambda}_m \frac{d}{dt}\Big\lfloor_{t=0}|\Phi_t(\EtildeVol)|
	\end{equation}
	where $\{\Phi_t\}_{|t|<1}$ is one-parameter diffeomorphism associated with $T \in C^{\infty}_c(\Rdim;\Rdim)$, that is, $\Phi_t(x) \coloneqq x+tT(x)$ for any $x \in \Rdim$ and $|t|<1$. Then, substituting the identity function $\mathrm{Id}: x \mapsto x$ with $T$ (by using a suitable cut-off function or by approximation), we obtain, from \eqref{eulerLagrangeEnergy},
	\begin{equation}\label{equationEulerLagrangeScaledMinmizer}
		 \dimension|B_1| \, \widetilde{\lambda}_m  = (\dimension-s)\Ps(\EtildeVol) + \int_{\EtildeVol} \nabla g_m(x) \cdot x \,dx + N \int_{\widetilde{E}_m} g_m(x) \,dx.
	\end{equation}
	From \cref{thm:Uniform_closeness_of_minimizers_to_ball} and the assumption on $g$, we obtain
	 \begin{align}
	 	\left| \int_{\EtildeVol} \nabla g_m(x) \cdot x \,dx + N \int_{\widetilde{E}_m} g_m(x) \,dx \right| \leq C(\dimension,s,\potential) \, m^{\frac{s}{\dimension}} \label{estimateIntegralBoundary}
	 \end{align}
	 
	for any $m \in (0,\,m_0)$ where $C=C(\dimension,s,\potential)>0$ is a constant and $\Rzero$ is given in \cref{prop:properties_of_minimizers} depending only on $\dimension$, $s$, and $\potential$. Then, from \eqref{equationEulerLagrangeScaledMinmizer} and \eqref{estimateIntegralBoundary} and by the minimality of $\EtildeVol$, we obtain
	\begin{equation}
		|\widetilde{\lambda}_m| \leq (\dimension|B_1|)^{-1} \left( (\dimension-s)\Ps(B_1) + C(\dimension,s,\potential)m^{\frac{s}{\dimension}} \right) \label{estimateLagrangeMultiplier}
	\end{equation}
	for any $m \in (0,\,m_0)$. Therefore, from \eqref{representationMinimizerNearBall02} and \eqref{estimateLagrangeMultiplier}, we obtain
	\begin{equation}\label{representationMinimizerNearBall03}
		\|\phi_m\|_{C^{2,\alpha}(\partial B_1)} \leq C_3(\dimension,s,\potential) < \infty
	\end{equation}
	for any $m \in (0,\,m_0)$ and some constant $C_3>0$ independent of $m$.

	Next, because of the continuous embedding $C^{2,\alpha}(\overline{\Omega}) \subset C^{2,\beta}(\overline{\Omega}) \subset C^0(\overline{\Omega})$ for any smooth bounded open set $\Omega$ and $\beta < \alpha$, we may observe that the following interpolation inequality is valid (see, for instance, \cite[Corollary 1.2.19 and Corollary 1.2.7]{Lun95}): let $\alpha \in (0,\,1)$. Then, for any $\beta \in (0,\,\alpha)$, there exist $\theta(\alpha,\beta) \in (0,\,1)$ and a constant $C=C(\dimension,\alpha,\beta)>0$ such that 
	\begin{equation*}
		\|u\|_{C^{2,\beta}} \leq C\, \|u\|_{C^{0}}^{\theta(\alpha,\beta)} \, \|u\|_{C^{2,\alpha}}^{1-\theta(\alpha,\beta)} 
	\end{equation*}
	for any $u \in C^{2,\alpha}$. From this interpolation inequality, we deduce that, for $\beta \in (0,\,\alpha)$, there exist constants $\theta(\alpha, \beta) \in (0,\,1)$ and $\widetilde{C}_3(\dimension,\alpha,\beta)>0$ such that  
	\begin{equation*}
		\|\phi_m\|_{C^{2,\beta}(\partial B_1(0))} \leq C_3(\dimension,\alpha,\beta) \|\phi_m\|_{L^{\infty}(\partial B_1(0))}^{\theta(\alpha,\beta)} \, \|\phi_m\|_{C^{2,\alpha}(\partial B_1(0))}^{1-\theta(\alpha, \beta)}
	\end{equation*}
	for $m \in (0,\,m_0)$. Therefore, from \eqref{representationMinimizersNearBall} and \eqref{representationMinimizerNearBall03}, we finally obtain that there exist constants $\widetilde{m}_0 = \widetilde{m}_0(\dimension,s,\potential)>0$ and $C_4(\dimension,s,\potential,\alpha)>0$ such that
	\begin{equation}
		\|\phi_m\|_{C^{2}(\partial B_1(0))} \leq C_4(\dimension,s,\potential,\alpha) \,  m^{\frac{\theta(\alpha)s^2}{2\dimension^2}}
	\end{equation}
	for any $m \in (0,\,\widetilde{m}_0)$. This implies that $\partial \EtildeVol$ is close to $\partial B_1(0)$ in $C^2$-sense and, in particular, $\EtildeVol$ is convex for sufficiently small $m>0$ because the second fundamental form of $\partial \EtildeVol$ pointwisely converges to that of $\partial B_1(0)$ which is positive definite.
\end{proof}

\begin{remark}\label{rem:Radial_Convex_Potential}
	In the proof of \cref{lemmaRegularityConvexityMinimizers}, we have observed the convexity of minimizers for small volumes, assuming that $\potential$ is locally Lipschitz, coercive, and $\inf_{\Rdim}\potential = \potential(0) = 0$. In addition, if we further assume that $\potential$ is convex and radially symmetric, that is, there exists a convex function $G: [0,\,\infty) \to \R$ such that $\potential(x) = G(|x|)$ for any $x \in \Rdim$, then we can show that the unique minimizer of $\EucEnergysVolScaledPotential$ with volume $|B_1|$ for $m>0$ is the ball up to negligible sets.

	Indeed, from the isoperimetric inequality of $\Ps$ and the symmetric rearrangement (see, for instance, \cite{LieLos01}), we have
	\begin{equation*}
		\Ps(E^*) \leq \Ps(E), \quad \int_{E^*}\volScaledPotential(x)\integralde x \leq \int_{E}\volScaledPotential(x)\integralde x
	\end{equation*}  
	for any measurable set $E \subset \Rdim$, where we denote by $E^*$ the open ball of radius equal to $|E|^{1/\dimension}|B_1|^{-1/\dimension}$ centred at the origin.
\end{remark}

\begin{corollary}\label{res:Lagrange_Multipl}
Under the assumptions of \cref{lemmaRegularityConvexityMinimizers}, for every $\volume>0$ the following Euler-Lagrange equation holds:
	\begin{equation}\label{eq:EulLag}
		H_{s,\partial E_m} + g = \lambda_m \quad \text{on $\partial E_m$},
	\end{equation}
	where $\lambda_m$ is a Lagrange multiplier associated to the volume constraint such that
	\begin{equation}\label{eq:Lag_Mult_Infinity}
	\lim\limits_{\volume\to 0}\lambda_\volume=+\infty.
	\end{equation}
\end{corollary}
\begin{proof}
 By making the same computation as in \eqref{equationEulerLagrangeScaledMinmizer}, we can show that the identity 
	\begin{equation}\label{equationLagrangeMultiplierOriginalProblem}
		N m \lambda_m = (N-s)\Ps(E_m) + \int_{E_m} \nabla g(x) \cdot x \,dx + N \int_{E_m} g(x) \,dx
	\end{equation}
	holds for $m>0$ and that $\lambda_m \to +\infty$ as $m \downarrow 0$.
\end{proof}

\begin{remark}[$H_s$-bubble problem]

	The classical $H$-bubble problem, which was originally raised by S.-T.Yau \cite[Problem 59]{Yau82No02}, is formulated as follows: given a function $H: \Rdim \to \R$, the question is to find an immersed hypersurface $\mathcal{M} \subset \Rdim$ such that its mean curvature at $p \in \mathcal{M}$ is equal to $H(p)$.
	\eqref{MinimizationNonlocalProblem} is related to a nonlocal version of this geometric problem.
	Using the same notation as in \cref{lemmaRegularityConvexityMinimizers} and \cref{res:Lagrange_Multipl}, we observe that if the function $(0,+\infty)\ni\volume\mapsto\lambda_m$ were continuous in a neighbourhood of zero, this would imply that it would be also surjective when restricted to a neighbourhood of zero (thanks to \eqref{eq:Lag_Mult_Infinity}). As a consequence, there would exist $\bar{\lambda}>0$ such that we would be able to solve a fractional version of the $H$-bubble problem, that is: we might find a hypersurface $\partial E$ satisfying the equation
	\begin{equation*}
	\sMeanCurvature{E} + \potential = \lambda \quad \text{on $\partial E$}, 
	\end{equation*}
	for every $\lambda>\bar{\lambda}$. 
	
	Regarding a sufficient condition for the continuity of the Lagrange multiplier $\lambda_m$, we remark that, if for $m_0>0$ sufficiently small a set $E_{m_0}$ with $|E_{m_0}|=m_0>0$ is the unique minimizer of $\capE_{s,\potential}$ among sets with volume $m_0$, then the Lagrange multiplier $\lambda_m$ is continuous in a neighbourhood of zero( with respect to $m>0$). 
	
	Indeed, to prove this we first take any sequence $\{m_i\}_{i} \subset (0,\,\infty)$ converging to a fixed $m_0 \in (0,\,\infty)$.	
	We observe, from the minimality of $E_{m_i}$, that $\sup_i\Ps(E_{m_i})<\infty$. Then, from the compactness of $P_s$, the uniqueness of $E_{m_0}$, and the coercivity of $g$, we deduce that $E_{m_i} \to E_{m_0}$ in $L^1$-sense. Thus, from the minimality of $E_{m_0}$, the uniform boundedness of $E_{m_i}$ and by using the dominated convergence theorem, we have
	\begin{equation}\label{continuityPotentialTerm}
		\int_{E_{m_i}}g(x) \,dx \xrightarrow[i \to \infty]{} \int_{E_{m_0}}g(x)\,dx.
	\end{equation}
	From the minimality of $E_{m_i}$, we can also obtain the continuity of $\Ps$ at $m_0$, that is,
	\begin{equation*}
		\Ps(E_{m_i}) \xrightarrow[i \to \infty]{} \Ps(E_{m_0}).
	\end{equation*}
	Indeed, from the minimality of $E_{m_i}$, setting $F_{m_i} \coloneqq \left(\frac{m_i}{m_0}\right)^{1/N} E_{m_0}$, we have
	\begin{align}
		E_{s,g}[m_i] &\coloneqq \inf\{\capE_{s,\potential}(E) \mid |E|=m_i\} \nonumber\\
		&\leq \capE_{s,\potential}(F_{m_i}) \nonumber\\
		&= \left( \frac{m_i}{m_0} \right)^{\frac{N-s}{N}}P_s(E_{m_0}) + \frac{m_i}{m_0} \int_{E_{m_0}} g\left( \left(\frac{m_i}{m_0}\right)^{\frac{1}{N}}x \right)\,dx \nonumber
	\end{align}
	for any $i \in \mathbb{N}$. Thus, from the minimality of $E_{m_0}$, we obtain that
	\begin{equation*}
		\limsup_{i \to \infty}E_{s,g}[m_i] \leq E_{s,g}[m_0].
	\end{equation*}
	In the same way, we can also obtain that $\liminf_{i\to\infty}E_{s,g}[m_i] \geq E_{s,g}[m_0]$. From this, we derive the continuity of $E_{s,g}[m]$ at $m_0$. Therefore, from the minimality of $E_{m_i}$ and $E_{m_0}$ and from \eqref{continuityPotentialTerm}, we obtain that $P_s$ is continuous at $m_0$. 
	
	Finally, from the assumption on $g$ and the uniform boundedness of $E_{m_i}$, we can also obtain that the potential term in \eqref{equationLagrangeMultiplierOriginalProblem} is continuous at $m_0$, that is, 
	\begin{equation*}
		\int_{E_{m_i}} \nabla g(x) \cdot x \,dx + N \int_{E_{m_i}} g(x) \,dx \xrightarrow[i \to \infty]{} \int_{E_{m_0}} \nabla g(x) \cdot x \,dx + N \int_{E_{m_0}} g(x) \,dx.
	\end{equation*} 
	In conclusion, from \eqref{equationLagrangeMultiplierOriginalProblem}, we obtain the continuity of $\lambda_m$.
\end{remark}

\bibliographystyle{plain}
\bibliography{biblio}

\end{document}